\documentclass{amsart}

\usepackage[text={14cm,23cm},centering,vcentering]{geometry}
\usepackage{url}

\usepackage[english]{babel}
\usepackage{anyfontsize}
\usepackage{microtype}
\usepackage{amsmath,amsfonts,amssymb,amsthm}
\usepackage{dsfont}
\usepackage{stmaryrd} 
\usepackage{relsize}
\usepackage[svgnames]{xcolor}
\usepackage{mathrsfs}
\usepackage{mdframed}
\usepackage{subcaption}
\usepackage{enumitem}
\usepackage{ tipa }

\usepackage{graphicx}

\usepackage{float}
\usepackage{framed}

\usepackage{pgf, tikz}
\usetikzlibrary{arrows,tikzmark}
\usepackage{pgfplots}
\usepackage{pgfplotstable}
\usepackage{tkz-base}
\usepackage{comment}

\usepackage{hyperref}
\hypersetup{
    colorlinks=true, 
    linktoc=all,     
    colorlinks,
    citecolor=black,
    filecolor=black,
    linkcolor=black,
    urlcolor=black
}

\newtheorem{thmA}{Theorem}

\theoremstyle{plain}
\newtheorem{theoreme}{Theorem}[section]
\newtheorem{prop}[theoreme]{Proposition}
\newtheorem{lem}[theoreme]{Lemma}
\newtheorem{cor}[theoreme]{Corollary}

\newtheorem*{conj}{Conjecture}
\newtheorem{ques}[theoreme]{Problem}

\theoremstyle{definition}
\newtheorem{rem}[theoreme]{Remark}
\newtheorem{defi}[theoreme]{Definition}

\newcommand{\Sum}{\mathlarger{\sum}}
\newcommand{\Prod}{\mathlarger{\prod}}

\newcommand{\C} {\mathbb{C}}

\newcommand{\Z} {\mathbb{Z}}
\newcommand{\N} {\mathbb{N}}

\newcommand{\Q} {\mathbb{Q}}
\newcommand{\F} {\mathbb{F}}

\newcommand{\T} {\mathbb{T}}

\newcommand{\oo}{\mathscr{O}_L}

\newcommand{\ens}{ \ \vert \ }
\newcommand{\sgn}{ \operatorname{sgn} }

\newcommand{\ord}{\operatorname{ord}}
\newcommand{\Lie}{\operatorname{Lie}}
\newcommand{\Log}{\operatorname{Log}}
\newcommand{\id}{\operatorname{id}}

\newcommand{\floor}[1]{\left\lfloor #1 \right\rfloor}

\usepackage{etoolbox}
\AtBeginEnvironment{enumerate}{\strut}
\AtBeginEnvironment{itemize}{\strut}

\pgfplotsset{
	compat=1.18,
	every axis/.append style={
		height=\pgfkeysvalueof{/pgfplots/width}*(((\pgfkeysvalueof{/pgfplots/ymax})+((-1)*(\pgfkeysvalueof{/pgfplots/ymin})))/((\pgfkeysvalueof{/pgfplots/xmax})+((-1)*(\pgfkeysvalueof{/pgfplots/xmin}))))+0.1pt, 
		axis x line=bottom,
		axis y line = left,
		axis lines=middle,							
		axis line style={-{Stealth}},				
		scale only axis,							 
		enlargelimits={0.1},
		grid style={densely dashed,line width=0.4pt,draw=black!30},
		minor grid style={line width=0.25pt,draw=black!20},
		axis equal,
		legend pos=north east,
	},
		every axis plot post/.append style={
		semithick,
	},
}

\tikzset{
    styleaxis/.style={
        circle,
        minimum size=2pt,	
    	inner sep=1.7pt,
    	fill,
    	fill opacity=1,
    	line width=1pt
	},
}

\pgfkeys{/pgf/number format/use comma,/pgf/number format/.cd,1000 sep={\,}}

\title{A $P$-adic class formula for Anderson $t$-modules}

\begin{document}
\author{Alexis Lucas}

\address{
Normandie Université, 
Université de Caen Normandie - CNRS, 
Laboratoire de Mathématiques Nicolas Oresme (LMNO), UMR 6139,  
14000 Caen, France.
}
\email{alexis.lucas@unicaen.fr}
\title{A $P$-adic class formula for Anderson $t$-modules}

\date{\today}


\parskip 2pt

\begin{abstract}
In 2012, Taelman proved a class formula for $L$-series associated to Drinfeld $\mathbb F_q[\theta]$ modules and considered it as a function field analog of the Birch and Swinnerton-Dyer conjecture. Since then, Taelman's class formula has been generalized to the setting of Anderson $t$-modules. Let $P$ be a monic prime of $\mathbb F_q[\theta]$, we define the $P$-adic $L$-series associated with Anderson $t$-modules and prove a $P$-adic class formula à la Taelman linking a $P$-adic regulator, the class module and a local factor at $P$. Next, we extend this result to the multi-variable setting à la Pellarin. Finally, we give some applications to Drinfeld modules defined over $\mathbb F_q[\theta]$ itself.
\end{abstract}


\maketitle

\tableofcontents

\section{Introduction}
\subsection{Class formula à la Taelman}
In 2010, Taelman introduced the notion of $L$-series associated to Drinfeld $\F_q[\theta]$-modules and conjectured a class formula, see \cite[Conjecture 1]{taelman2010}. He \cite{Taelman} proved this class formula in 2012, also considered as a function field analogue of the Birch and Swinnerton-Dyer conjecture. These results were generalised by Fang \cite{Fang} and Demeslay \cite{Flo} in the context of Anderson $t$-modules that are Drinfeld modules of higher dimension. Finally, Anglès, Tavares Ribeiro and Ngo Dac \cite{Admissible} proved the class formula for a general ring $A$ in the context of admissible Anderson $A$-modules, including in particular all Drinfeld $A$-modules.

The objective of the present article is to construct $P$-adic analogs of these $L$-series associated to Anderson $t$-modules. We call them $P$-adic $L$-series, and prove a $P$-adic class formula à la Taelman. We then extend these results to the setting of variables following the work of Anglès, Pellarin and Tavares Ribeiro \cite{federico} and Pellarin \cite{pellarin}.

The key ingredient will be the notion of $z$-deformation introduced by Anglès, Tavares Ribeiro \cite{units} as well as the introduction of evaluation of $z$ not only at $z=1$ but at elements of $\overline{\F}_q$.

We then study the vanishing of the $P$-adic $L$-series at $z=1$.
Finally, we investigate in detail the case of $\F_q[\theta]$-Drinfeld modules defined over $\F_q[\theta]$.

\subsection{Main results}
Let us give more precise statements of our results.

Let $\F_q$ be a finite field with $q$ elements and $\theta$ an indeterminate over $\F_q$.
Let us consider $A=\F_q[\theta]$ and let $K=\F_q(\theta)$ be the rational function field. Let $L/K$ be a finite field extension of degree $n$. We denote by $\oo$ the integral closure of $A$ in $L$. We consider the valuation $v_\infty$ of $K$ normalized such that $v_\infty(\theta^{-1})=1$. Let $K_\infty$ be the completion of $K$ with respect to this valuation and we set $L_\infty=L\otimes_K K_\infty$. Let $\tau:x\mapsto x^q$ be the Frobenius map.

If $M$ is an $A$-module having a finite number of elements, we denote by $[M]_A$ the monic generator of the Fitting ideal of $M$.

An Anderson $t$-module $E$ of dimension $d$ defined over $\oo$ is a non-constant $\F_q$-algebra homomorphism $E:A\rightarrow M_d(\oo)\{\tau\}$ such that if $a\in A$ and $E_a=\Sum\limits_{i=0}^{r_a} E_{a,i}\tau^i$ then we require $(E_{a,0}- aI_d)^d=0$. We denote by $\partial_E:A\rightarrow M_d(\oo)$ the homomorphism of $\F_q$-algebras $\partial_E(a)=E_{a,0}$. If $B$ is an $\oo$-algebra, then we can define two $A$-module structures on $B^d$: the first is denoted by $E(B)$ where $A$ acts on $B^d$ via $E$, and the second is denoted by $\operatorname{Lie}_E(B)$ where $A$ acts on $B^d$ via $\partial_E$. 

There exists a unique series $\exp_E\in M_d(L)\{\{\tau\}\}$, called the exponential series, such that $\exp_E \partial_E(\theta)=E_\theta \exp_E$. Moreover, $\exp_E$ converges on $\operatorname{Lie}_E(L_\infty)$.

The key notion will be the notion of $z$-deformation introduced by Anglès and Tavares Ribeiro \cite{units}. Let $z$ be a new variable such that $\tau(z)=z$. Set $\widetilde{A}=\F_q(z)A$, $\widetilde{K}=\F_q(z)K$, $\widetilde{\oo}=\F_q(z)\oo$ and $\widetilde{K_\infty}=\F_q(z)((\theta^{-1}))$. We set $\widetilde{L_\infty}=L\otimes_K \widetilde{K}_\infty$.
We consider $\widetilde{E}$ the $z$-twist of $E$, introduced in \cite{units}, that is the homomorphism of $\F_q(z)$-algebras $\widetilde{E}:\widetilde{A}\rightarrow M_d(\widetilde{\oo})\{\tau\}$ given by
$$\widetilde{E}_a=\Sum\limits_{i=0}^{r_a} E_{a,i}z^i\tau^i, \text{ for all } a\in A.$$
We also have an exponential series $\exp_{\widetilde{E}}$ associated with $\widetilde{E}$ that converges on $\operatorname{Lie}_{\widetilde{E}}(\widetilde{L_\infty})$.

 Taelman showed that the module of $z$-units 
 $$U(\widetilde{E};\widetilde{\oo})=\{x\in \operatorname{Lie}_{\widetilde{E}}(L_\infty)\ens \exp_{\widetilde{E}}(x)\in \widetilde{E}(\widetilde{\oo})\}$$ is an $\widetilde{A}$-lattice in $\operatorname{Lie}_{\widetilde{E}}(\widetilde{L_\infty})$ and that the Taelman unit module 
 $$U(E;\oo)=\{x\in \operatorname{Lie}_E(L_\infty)\ens \exp_E(x)\in E(\oo)\}$$ is an $A$-lattice in $\operatorname{Lie}_E(L_\infty)$. Moreover, he proved that the class module
 $$H(E;\oo)=\dfrac{E(L_\infty)}{E(\oo)+\exp_E(\operatorname{Lie}_E(\oo))}$$ is finite. The local factors associated with $\widetilde{E}$ and $E$ at a monic prime polynomial $Q$ are respectively $$z_Q(\widetilde{E}/\widetilde{\oo})=\dfrac{\left[\operatorname{Lie}_{\widetilde{E}}(\widetilde{\oo}/Q\widetilde{\oo})\right]_{\widetilde{A}}}{[\widetilde{E}(\widetilde{\oo}/Q\widetilde{\oo})]_{\widetilde{A}}}\in \widetilde{K}^\ast \text{ and } z_Q(E/\oo)=\dfrac{\left[\operatorname{Lie}_E(\oo/Q\oo)\right]_A}{[E(\oo/Q\oo)]_A}\in K^\ast.$$ Set $m=dn$ and consider $\mathscr{C}$ an $A$-basis of $\operatorname{Lie}_E(\oo)$. Fix $(u_1(z),\ldots, u_m(z))$ an $\widetilde{A}$-basis of $U(\widetilde{E};\widetilde{\oo})$ and $(u_1,\ldots, u_m)$ an $A$-basis of $U(E;\oo)$.  Demeslay proved in \cite{Flo} that the following product converges in $\T_z(K_\infty)$, the completion of $K_\infty[z]$ with respect to the Gauss norm:
$$L(\widetilde{E}/\widetilde{\oo})=\Prod\limits_{Q}z_Q(\widetilde{E}/\widetilde{\oo})$$ where the product runs over all the monic irreducible polynomials $Q$ of $A$. We call this product the $L$-series associated with $\widetilde{E}$ and $\widetilde{\oo}$. By evaluation at $z=1$ we obtain:
$$L(E/\mathscr{O}_L)=\operatorname{ev}_{z=1}L(\widetilde{E},\widetilde{\mathscr{O}}_L)=\Prod\limits_{Q}z_Q(E/\oo)\in K_\infty^\ast.$$
We call this product the $L$-series associated with $E$ and $\oo$. Demeslay \cite{Flo} proved the following class formulas à la Taelman
$$L(\widetilde{E}/\widetilde{\oo})=\dfrac{\det_{\mathscr{C}}(u_1(z),\ldots, u_m(z))}{\operatorname{sgn}(\det_{\mathscr{C}}(u_1(z),\ldots, u_m(z)))}$$ and
$$L(E/\oo)=\dfrac{\det_\mathscr{C}(u_1,\ldots, u_m)}{\operatorname{sgn}(\det_\mathscr{C}(u_1,\ldots, u_m))}\left[H(E;\oo)\right]_A.$$

We consider the $P$-adic setting.
Let $P$ be an irreducible monic polynomial of $A$ and $v_P$ its associated valuation on $K$ such that $v_P(P)=1$. We consider $K_P\simeq \F_{q^{\deg (P)}}((P))$ (resp. $A_P$) the completion of $K$ (resp. $A$) with respect to $P$. Consider the Tate algebra in the variable $z$,
$\T_z(K_P)$, that is the completion of $K_P[z]$ with respect to the Gauss norm.
Our first main result is the construction and the convergence of the following $P$-adic $L$-series. The key argument will be the evaluation at $z=\zeta\in \overline{\F}_q$, see Subsections \ref{3.4} and \ref{extensionP}.
\begin{thmA}[Theorem \ref{th:paspole}] The following product converges in $\T_z(K_P)$  $$L_P(\widetilde{E}/\widetilde{\oo})=\Prod\limits_{Q\neq P}z_Q(\widetilde{E}/\widetilde{\oo})$$ where the product runs over all the monic irreducible polynomials $Q$ of $A$ different from $P$.
 \end{thmA}

We then define a $P$-adic logarithm $\Log_{E,P}$ which converges on $\{x\in \mathscr{O}_L^d \ens v_P(x)\geq 0\}$ and a $P$-adic regulator associated with the unit module as follows. Let $(u_1,\ldots,u_m)$ be an $A$-basis of the unit module. We then define the $P$-adic regulator of the unit module by:
 $$R_P(U(E;\oo))=\dfrac{\det_{\mathscr{C}}\left(\Log_{E,P}(\exp_E (u_1)),\ldots, \Log_{E,P}(\exp_E (u_m))\right)}{\operatorname{ sgn } (\det_{\mathscr{C}}(u_1,\ldots,u_m))}\in K_P.$$
 The construction does not depend on $\mathscr{C}$ nor on the choice of the basis of the unit module. We do the same with the variable $z$. We then prove the following $P$-adic class formula à la Taelman.
 \begin{thmA}[Theorem \ref{maincf}] Let $E$ be an Anderson $t$-module defined over $\oo$. Then we have the $P$-adic class formula for $\widetilde{E}$:
    $$z_P(\widetilde{E}/\widetilde{\oo})L_P(\widetilde{E}/\widetilde{\oo})=R_P(U(\widetilde{E};\widetilde{\oo}))$$
    and the class formula for $E$:
    $$z_P(E/\oo)L_P(E/\oo)=R_{P}(U(E;\oo))\left[H(E;\oo))\right]_{A}.$$
\end{thmA}

 A major difference from the $\infty$-adic case is that our $P$-adic $L$-series obtained will vanish at $z=1$ in certain cases that we are able to characterize. Set $U(E;P\oo)=\{x\in \operatorname{Lie}_E(L_\infty)\ens \exp_E(x)\in E(P\oo)\}$ and $\mathscr{U}(E;P\mathscr{O}_{L})=\exp_E(U(E;P\mathscr{O}_{L}))$ which is provided with an $A_P$-structure module
\begin{thmA}[Proposition \ref{main3} and Conjecture \ref{leopoldt}] We have the following assertions.
    \begin{enumerate}
        \item If the exponential map $\exp_E:L_{\infty}^d\rightarrow L_{\infty}^d$ is not injective, then $L_P(E/\mathscr{O}_{L})=0$.

        \item If the $A$-rank of $\exp_E(U(E;\mathscr{O}_{L}))$ and the $A_P$-rank of $\mathscr{U}(E;P\mathscr{O}_{L})$ coincide, then the two following assertions are equivalent:
        \begin{enumerate}
            \item $L_P(E/\oo)\neq0$,
            \item $ \exp_E:L_{\infty}^d\rightarrow L_{\infty}^d \text{ is injective}$.
        \end{enumerate}
    \end{enumerate}
\end{thmA}

We extend the previous results to the multi-variable setting in the spirit of the work of Anglès, Pellarin and Tavares Ribeiro \cite{federico}, Demeslay \cite{Flo} and Pellarin \cite{pellarin}. Consider $s\geq 1$ an integer and $t_1,\ldots,t_s$ new variables. Set $k=\F_q(t_1,\ldots, t_s)$, $A_s=k[\theta]$, $\mathscr{O}_{L,s}=k\oo$ and $K_{s,P}=\F_{q^{\deg (P)}}(t_1,\ldots, t_s)((P))$. Let $s\geq 0$ be a non-negative integer and $E$ be an Anderson $A_s$-module defined over $\mathscr{O}_{L,s}$. We consider the Tate algebra with multi-variable $\T_s(K_P)$ that is the completion of $K_P[t_1,\ldots, t_s]$. We prove the following result.
\begin{thmA}[Theorem \ref{main2} and Theorem \ref{main0}]
    \begin{enumerate}
\item  The infinite product 

$$L_P({E}/{\mathscr{O}}_{L,s})=\Prod\limits_{Q\neq P } z_Q(E/\mathscr{O}_{L,s})$$ where $Q$ runs through the monic primes of $A$ different from $P$, converges in $K_{s,P}$ and we have the class formula:
$$z_P(E/\mathscr{O}_{L,s})L_P({E}/{\mathscr{O}}_{L,s})=R_P(U({E};{\mathscr{O}}_{L,s}))\left[H(E;\mathscr{O}_{L,s})\right]_{A_s}.$$
\item If $E$ is defined over $\F_q[t_1,\ldots, t_s]\oo$, then 
$L_P(E/\mathscr{O}_{L,s})\in \T_s(K_P)$.
\end{enumerate}
\end{thmA}

The key point for proving the second assertion will be the successive evaluation at $t_i=\zeta_i\in \overline{\F_q}$ so using techniques from Subsections \ref{3.4} and \ref{extensionP}.

Finally, in Section \ref{section:facile}, we give a detailed application of the previous theorems and obtain bounds on the vanishing order at $z=1$ of the $P$-adic $L$-series.
\begin{thmA}[Proposition \ref{5}] Consider $\phi$ an $A$-Drinfeld module defined over $A$ of rank $r$. Then the vanishing order at $z=1$ of the $P$-adic $L$-series is less than or equal to $r$.
\end{thmA}

The $P$-adic $L$-series is studied in detail in a work of Caruso, Gazda and the author in the context of $A$-Drinfeld modules defined over $A$, see \cite{cgl}.

\subsection{Some remarks}
Recently, Anglés \cite{An} defined the $P$-adic $L$-series $L_P(\widetilde{C}/\widetilde{\oo})$ associated with the Carlitz module defined over $\oo$. He was able to prove that the $P$-adic $L$-series $L_P(\widetilde{C}/\widetilde{\oo})$ is a meromorphic series without pole at $z=1$, see \cite[Theorem 6.6]{An}. In particular he defined the $P$-adic $L$-series as the limit at $z=1$ of this series which is an element of $K_P$ and proved a $P$-adic class formula, see \cite[Theorem 6.7]{An}. The present work generalizes this result for Anderson $t$-modules, including all Drinfeld modules, by proving that the $P$-adic $L$-series is an element of $\T_z(K_P)$ so that we can evaluate at $z=1$.

\paragraph{\textbf{Acknowledgments}}
This work was carried out during my PhD thesis under the supervision of Tuan Ngo Dac and Floric Tavares Ribeiro. I would like to thank them warmly for the interesting discussions, the sharing of ideas and the time devoted to this work. I would also like to thank Bruno Anglès for the interesting discussions and remarks.

\section{Notation and background}\label{section:notation}

\theoremstyle{plain}
\subsection{Notation}
We keep the notation in the Introduction and we introduce the following notation.\\ 
$\bullet$ $\C_\infty$: the completion of a fixed algebraic closure of $K_\infty$,\\ 
$\bullet$ $\tau:\C_\infty\rightarrow \C_\infty$ the Frobenius endomorphism, \\ 
$\bullet$ $M_d(R)=M_{d\times d}(R)$, for a ring $R$ the left $R$-module of $d\times d$ matrices,\\
$\bullet$ $I_d$: the identity matrix of $M_d(R)$.

Let us fix an integer $d\geq 1$ and $B$ an $\F_q$-algebra. 
If $M=(m_{i,j})$ is a matrix with coefficients in $\C_{\infty}$ and $k\in \N$, we set $\tau^k (M)=M^{(k)}$ to be the matrix whose $ij$-entry is given by $\tau^k(m_{i,j})^{(k)}=m_{i,j}^{q^k}$. We denote by $M_{d}(B)\{\tau\}$ the non-commutative ring of twisted polynomials in $\tau$ with coefficients in $M_d(B)$ equipped with the usual addition and the commutation rule $\tau^k M = M^{(k)} \tau^k$ for all $k\in \N$ and all $M \in M_d(B)$. Let $M_{d}(B)\{\{\tau\}\}$ be the non-commutative ring of twisted power series in $\tau
$ with coefficients in $M_d(B)$.

If $k$ is a field containing $\F_q$, we set $(kK)_\infty=k\hat{\otimes}_{\F_q}K_\infty=k((\frac{1}{\theta})).$ If $x\in  (kK)_\infty^\times$, we can write $x$ uniquely as $x =\Sum\limits_{n\geq N} x_n \dfrac{1}{\theta^n}$, $x_n\in k$ and $x_N\neq 0$. Then we call $x_N\in K$ the sign denoted by $\operatorname{sgn}(x)$. We say that such an $x\in (kK)_\infty$ is monic if $\operatorname{sgn}(x) = 1$.

\subsection{Fitting ideals}
We recall here some definitions about Fitting ideals of modules over Dedekind rings. Let $R$ be a Dedekind ring, and $M$ be a finite and torsion $R$-module. By the structure theorem, there exists $s\in \N$ and $I_1,\ldots ,I_s$ ideals of $R$ such that we have an isomorphism of $R$-modules 
$$M\simeq R/I_1\times \ldots \times R/I_s.$$
We then define the Fitting ideal of $M$ by 
$$\operatorname{Fitt}_R(M)=I_1\ldots  I_s.$$
We have the following properties that can be found in the appendix of \cite{fitting1} except the second one which appears in \cite[Corollary 20.5]{fitting2}.

\begin{prop}\label{fitt}
\begin{enumerate}

\item We have $\operatorname{Fitt}_R(M ) \subseteq\operatorname{Ann}_R(M)=\{ x\in R \ens x.m=0 \ \forall m\in M\}$. 
\item If $0 \rightarrow M_1 \rightarrow M \rightarrow M_2 \rightarrow 0$ is exact, then $\operatorname{Fitt}_R(M_1 ) \operatorname{Fitt}_R (M_2 )=\operatorname{Fitt}_R (M )$.
\end{enumerate}
\end{prop}

\subsection{Lattices and Ratio of co-volumes}

We use the following notation from \cite[Section 7.2.3]{F}. We fix $k$ a field containing $\F_q$ and recall that $(kK)_\infty=k\hat{\otimes}_{\F_q}K_\infty=k\left(\left(\frac{1}{\theta}\right)\right)$.
In what follows, we fix $V$ to be a finite-dimensional $(kK)_\infty$-vector space endowed
with the natural topology coming from $(kK)_\infty$.
\begin{defi} A sub-$k[\theta]$-module $M$ of $V$ is a $k[\theta]$-lattice in $V$ if $M$ is discrete in $V$ and if $M$ generates $V$ over $(kK)_\infty$.
\end{defi}

\begin{prop}Let $M$ be a sub-$k[\theta]$-module of $V$. The following are equivalent:
\begin{enumerate}
\item $M$ is a $k[\theta]$-lattice in $V$.
\item There exists a $(kK)_\infty$-basis $(v_1,\ldots , v_n)$ of $V$ such that $M$ is the free $k[\theta]$-module of basis $(v_1,\ldots ,v_n)$.
\end{enumerate}
\end{prop}

\begin{proof} See \cite[Proposition 7.2.3]{F}.
\end{proof}

Let $M$ and $M'$ be two $k[\theta]$-lattices in $V$. Let $\mathscr{B}$ and $\mathscr{B}'$ be $k[\theta]$-basis of $M$ and $M'$, respectively. The ratio of co-volumes of $M$ in $M'$ is then defined as

$$[M':M]_{k[\theta]}=\dfrac{\det_{\mathscr{B}'} \mathscr{B}}{\operatorname{sgn}(\det_{\mathscr{B}'} \mathscr{B})}\in (kK)_\infty^\ast.$$
Note that this is independent of the choices of $\mathscr{B}$ and $\mathscr{B}'$. The definition immediately implies that if $M_0,$ $M_1$ and $M_2$ are lattices in $V$, then
$$[M_0 : M_1]_{k[\theta]}[M_1 : M_2 ]_{k[\theta]} = [M_0 : M_2 ]_{k[\theta]}.$$
We also observe that for two lattices $M, M'$ in $V$ we have
$$[M' : M]_{k[\theta]}=[M : M']_{k[\theta]}^{-1}.$$

\section{The \texorpdfstring{$\infty$}{}-adic case}\label{section:infini}
From now on, let $L/K$ be a finite fields extension. Recall that we denote by: $\oo$ the integral closure of $A$ in $L$, $\mathscr{O}_{L}[z]\simeq \F_q[z]\otimes_{\F_q}\mathscr{O}_{L}$, $\widetilde{\mathscr{O}_{L}}=\F_q(z)\otimes_{\F_q}\mathscr{O}_{L}$, $\widetilde{L_\infty}=L\otimes_{\F_q} \widetilde{K_\infty}$. In this section, we extend the notion of the Taelman unit module and class module by twisting with some elements $\zeta\in \overline{\F_q}$.
\subsection{Anderson modules}

An Anderson $t$-module (or shortly a $t$-module) $E$ of dimension $d$ defined over $\oo$ is an $\F_q$-algebra homomorphism $E:A\rightarrow M_d(\oo)\{\tau\}$ such that if $a\in A$ and $E_a=\Sum\limits_{i=0}^{r_a} E_{a,i}\tau^i$, then we require $(E_{a,0}- aI_d)^d=0$ and that $\deg_\tau(E_\theta)> 0$. Let $E:A\rightarrow M_d(\mathscr{O}_L)\{\tau\}$ be a $t$-module of dimension $d\geq 1$, completely determined by the value at $\theta$:
$$E_\theta=\Sum\limits_{i=0}^{r}E_{\theta,i}\tau^i$$ with $E_{\theta,i}\in M_d(\mathscr{O}_L)$ and $(E_{\theta,0}-\theta I_d)^d=0$. For $a\in A$, if $E_a=\Sum\limits_{i=0}^{r_a}E_{a,i}\tau^i$, we define $\partial_E(a)=E_{a,0}$. Then the map $\partial_E:A\rightarrow M_d(A)$ is a homomorphism of $\F_q$-algebras.

We then consider the $z$-twist of $E$, introduced in \cite{units}, (remember that $\tau$ acts as the identity over $\F_q(z)$) denoted by $\widetilde{E}$ to be the homomorphism of $\F_q(z)$-algebras $\widetilde{E}:\widetilde{A}\rightarrow M_d(\widetilde{\oo})\{\tau\}$ given by:
$$\widetilde{E}_\theta=\Sum\limits_{i=0}^{r}E_{\theta,i}z^i\tau^i.$$

Recall the following notation taken from \cite{Admissible}.
Let $E$ be a $t$-module of dimension $d$ over $R$ an extension of $\F_q$ and let $B$ be an $R$-algebra. We can then define two $A$-module structures on $B^d$. The first is denoted $E(B)$ where $A$ acts on $B^d$ via $E$:
$$a.x=E_a(x)\in B^d \text{ for all } a\in A, x\in B^d.$$
The second is $\operatorname{Lie}_E(B)$ where $A$ acts on $B^d$ via $\partial_E$:
$$a.x=\partial_E(a)x \text{ for all } a\in A, x\in B^d.$$
We have the following results that can be found in \cite[Proposition 2.1.4]{Anderson}.

\begin{prop}\label{prop:exp}
 There exists a unique element $\operatorname{exp}_E\in M_d(L)\{\{\tau\}\}$ such that:
\begin{enumerate} \item $\operatorname{exp}_E \partial_E(a)= E_a\operatorname{exp}_E \ \text{ hold in } M_d(L)\{\{\tau\}\} \text{ for all } a \in A,$
 \item $\operatorname{exp}_E \equiv I_d \text{ mod } M_d(L) \{\{\tau\}\}\tau.$
 \end{enumerate}
 \end{prop}
 We call $\exp_E$ the exponential map associated with the $t$-module $E$, and denote this element by $\exp_E=\Sum\limits_{i=0}^{\infty}d_n\tau^n.$

\begin{prop}\label{prop:log}
There exists a unique element $\operatorname{log}_E\in M_d(L)\{\{\tau\}\}$ such that:
\begin{enumerate} \item $\operatorname{log}_E E_a=\partial_E(a) \operatorname{log}_E \text{ hold in } M_d(L)\{\{\tau\}\} \text{ for all } a \in A$,
 \item $\operatorname{log}_E \equiv I_d \text{ mod }M_d(L) \{\{\tau\}\}\tau .$
 \end{enumerate}
 In addition, we have the equalities in $M_d(L)\{\{\tau\}\}$:
 $$\log_E \exp_E=\exp_E \log_E =I_d.$$
\end{prop}
We call $\log_E$ the logarithm map associated to the $t$-module $E$, and we denote this element by $\log_E=\Sum\limits_{n=0}^{\infty}l_n\tau^n.$ We also have exponential and logarithm series for the $z$-twist of the $t$-module $\widetilde{E}$ which we denote by $\exp_{\widetilde{E}}$ and $\log_{\widetilde{E}}$ and given by:
$$\exp_{\widetilde{E}}=\Sum\limits_{n\geq 0} d_nz^n\tau^n\text{ and } \log_{\widetilde{E}}=\Sum\limits_{n\geq 0} l_nz^n\tau^n.$$
\subsection{Unit module and class module}

We consider an over-additive valuation $v_\infty$ on the-finite dimensional $K_\infty$-vector space $L_\infty$ (for example with respect to the choise of a basis of $L/K)$.
The key point is the next result from \cite[Theorem 4.6.9]{Goss}.

\begin{lem}\label{lem:expinfini} We have
$$\lim\limits_{i\rightarrow +\infty} \dfrac{v_\infty(d_i)}{q^i}=+\infty.$$
\end{lem}

\begin{cor}\label{cvei}
The exponential map $\exp_E$ converges on $\operatorname{Lie}_E(L_\infty)$ and induces a homomorphism of $A$-modules:
    $$\exp_E:\operatorname{Lie}_E(L_\infty)\rightarrow E(L_\infty).$$
    \end{cor}
We also have the convergence of $\exp_{\widetilde{E}}$ on $\operatorname{Lie}_{\widetilde{E}}(\widetilde{L_\infty})$ (resp. $\operatorname{Lie}_{\widetilde{E}}(\T_z(L_\infty))$) that induces a homomorphism of $\widetilde{A}$-modules (resp. $A[z]$):
$$\exp_{\widetilde{E}}:\operatorname{Lie}_{\widetilde{E}}(\widetilde{L_\infty})\rightarrow \widetilde{E}(\widetilde{L_\infty}).$$
(resp. $\exp_{\widetilde{E}}:\operatorname{Lie}_{\widetilde{E}}(\T_z(L_\infty))\rightarrow \widetilde{E}(\T_z(L_\infty))$). We can now define the Taelman unit module

$$U(E;\mathscr{O}_{L})=\{x\in \operatorname{Lie}_E(L_\infty)\ens \exp_E (x)\in E(\mathscr{O}_L)\}$$
provided with $A$-module structure, as well as the module of $z$-units:
$$U(\widetilde{E};\widetilde{\oo})=\left\{x\in \operatorname{Lie}_{\widetilde{E}}(\widetilde{L_\infty}) \ens \exp_{\widetilde{E}}(x)\in \widetilde{E}(\widetilde{\oo})\right\}$$
provided with $\widetilde{A}$-module structure,
 and the module of $z$-units at the integral level:
$$U(\widetilde{E};\mathscr{O}_L[z])=\left\{x\in \operatorname{Lie}_{\widetilde{E}}(\T_z(L_\infty)) \ens \exp_{\widetilde{E}}(x)\in\widetilde{E}(\mathscr{O}_L[z])\right\}$$
provided with $A[z]$-module structure. We also define the class module (introduced by Taelman in \cite{Taelman}):
$$H(E;\mathscr{O}_L)=\dfrac{E(L_\infty)}{E(\mathscr{O}_L)+\exp_{{E}}(\operatorname{Lie}_{{E}}(L_\infty))}$$
as well as the class module for the $z$-deformation
$$H(\widetilde{E};\widetilde{\oo})=\dfrac{\widetilde{E}(\widetilde{L_\infty})}{\widetilde{E}(\widetilde{\oo})+\exp_{\widetilde{E}}(\operatorname{Lie}_{\widetilde{E}}(\widetilde{L_\infty}))}$$
and finally the class module at the integral level
$$H(\widetilde{E};\mathscr{O}_L[z])=\dfrac{\widetilde{E}(\T_z(L_\infty))}{\widetilde{E}(\mathscr{O}_L[z])+\exp_{\widetilde{E}}(\operatorname{Lie}_{\widetilde{E}}(\T_z(L_\infty)))}.$$

Consider one of the following cases:\\
$\bullet$ $k_0=\F_q$, $\varphi=E$ and $(k_0L)_\infty=L_\infty$,\\
$\bullet$ $k_0=\F_q[z]$ and $\varphi=\widetilde{E}$,\\
$\bullet$ $k_0=\F_q(z)$, $\varphi=\widetilde{E}$ and $(k_0L)_\infty=\widetilde{L_\infty}$.  \\ 
We have the following result from \cite[Proposition 2.8]{Flo}.

\begin{prop}\label{prop:classand units}
    \begin{enumerate}
        \item The class module $H(\varphi;k_0\mathscr{O}_L)$ is a finite-dimensional $k_0$-vector space, so a finite and torsion $k_0A$-module.
        \item If $k_0$ is a field, then the module of units $U(\varphi;k_0\mathscr{O}_L)$ is a $k_0A$-lattice in $\operatorname{Lie}_{\varphi}((k_0L)_\infty)$.
    \end{enumerate}
\end{prop}
We also have the following result in \cite[Proposition 2.3]{units}.
\begin{prop} We have the following equality:
$$U(\widetilde{E};\widetilde{\oo})=\F_q(z)U(\widetilde{E};\mathscr{O}_L[z]).$$
\end{prop}
Consider the evaluation morphism:
$$\operatorname{ev}_{z=1}:\operatorname{Lie}_{\widetilde{E}}(\T_z(L_\infty))\rightarrow \operatorname{Lie}_E(L_\infty).$$
It induces an exact sequence of $A$-modules:
$$0\longrightarrow (z-1)\operatorname{Lie}_{\widetilde{E}}(\T_z(L_\infty))\longrightarrow \operatorname{Lie}_{\widetilde{E}}(\T_z(L_\infty)) \overset{\operatorname{ev}_{z=1}}\longrightarrow\operatorname{Lie}_E(L_\infty)\longrightarrow 0.$$
For all $x\in \operatorname{Lie}_{\widetilde{E}}(\T_z(L_\infty))$ we have $\operatorname{ev}_{z=1} (\exp_{\widetilde{E}} (x))=\exp_ E (\operatorname{ev}_{z=1} (x)).$ Moreover, if $f(z)\in \widetilde{L_\infty}$ belongs to the $\infty$-adic convergence domain of the logarithm map $\log_{\widetilde{E}}$, then we have $$\operatorname{ev}_{z=1}( \log_{\widetilde{E}} (f(z)))=\log_E(\operatorname{ev}_{z=1} (f(z))).$$

We recall the notion of Stark units introduced by B. Anglès and F. Tavares Ribeiro in \cite[section 2.5]{units}.

\begin{defi}The module of Stark units $U_{\operatorname{St}}(E;\oo)$ is defined by:
$$U_{\operatorname{St}}(E;\mathscr{O}_L)=\operatorname{ev}_{z=1} U(\widetilde{E};\mathscr{O}_L[z]).$$
\end{defi}
Given the compatibility between the exponential and the evaluation morphism, $U_{\operatorname{St}}(E;\mathscr{O}_L)$ is a sub-$A$-module of $U(E,\mathscr{O}_L). $ We have the following result from \cite[Theorem 1]{units}.

\begin{theoreme}
    The $A$-module $U_{\operatorname{St}}(E;\mathscr{O}_L)$ is an $A$-lattice in $\operatorname{Lie}_E(L_\infty)$.
\end{theoreme}

\subsection{The \texorpdfstring{$L$}{} series}

For a monic prime $P$ of $A$, we define the local factor at $P$ associated with $E$:
$$z_P(E/\mathscr{O}_L)=\dfrac{\left[\operatorname{Lie}_E(\mathscr{O}_L/P\mathscr{O}_L)\right]_{A}}{\left[E(\mathscr{O}_L/P\mathscr{O}_L)\right]_{A}}\in K$$
and the local factor at $P$ associated with $\widetilde{E}$:
$$z_P(\widetilde{E}/\widetilde{\oo})=\dfrac{\left[\operatorname{Lie}_{\widetilde{E}}(\widetilde{\oo}/P\widetilde{\oo})\right]_{\widetilde{A}}}{\left[\widetilde{E}(\widetilde{\oo}/P\widetilde{\oo})\right]_{\widetilde{A}}}\in \widetilde{K}.$$
We then define the $L$-series associated with $E$ and $\oo$ by the Eulerian product:
$$L(E/\oo)=\Prod\limits_{P\in A}z_P(E/\oo)$$ where $P$ runs through the monic primes of $A$, and the $L$-series associated with $\widetilde{E}$ and $\widetilde{\oo}$ by the Eulerian product:
$$L(\widetilde{E}/\widetilde{\oo})=\Prod\limits_{P\in A}z_P(\widetilde{E}/\widetilde{\oo}).$$
We have the convergence of the $L$-series and the class formula for $z$-deformation from \cite[Theorem 2.7]{Flo}.

\begin{theoreme}[Class formula for the $z$-deformation] The product defining $L(\widetilde{E}/\widetilde{\oo})$ converges in $\widetilde{K_\infty}^\ast$ and we have the formula 
\begin{equation}\label{clz}L(\widetilde{E}/\widetilde{\oo})=\left[\operatorname{Lie}_{\widetilde{E}}\left(\widetilde{\mathscr{O}_L}\right):U(\widetilde{E};\widetilde{\oo})\right]_{\widetilde{A}}.\end{equation}
\end{theoreme}
Adapting the proof of \cite[Corollary 7.5.6]{F} in the higher dimensional case we obtain that the polynomial $\left[\operatorname{Fitt}_{\widetilde{A}}(\widetilde{E}(\widetilde{\oo}/{P}\widetilde{\oo}))\right]_{\widetilde{A}}\in A[z]$ is a unit in $\T_z(K_\infty)$. We then obtain:
\begin{cor}\label{fittnonnul} The $L$-series $L(\widetilde{E}/\widetilde{\oo})$ converges in $\T_z(K_\infty)^{\times}.$
\end{cor}
We can evaluate the $L$-series at $z=1$:
$$L(E/\mathscr{O}_L)=\operatorname{ev}_{z=1}L(\widetilde{E}/\widetilde{\oo})= \Prod\limits_{Q}\dfrac{\left[\operatorname{Lie}_E(\mathscr{O}_L/Q\mathscr{O}_L)\right]_{A}}{\left[E(\mathscr{O}_L/Q\mathscr{O}_L))\right]_{A}}\in K_\infty^\ast $$ where $Q$ runs through the monic primes of $A$. We have the following class formula for $t$-modules obtained by Fang in \cite{Fang}, generalizing Taelman's class formula for Drinfeld modules.

\begin{theoreme}[Class formula for Anderson $t$-modules]\label{classformula}
    The product defining $L(E/\mathscr{O}_L)$ converges in $K_\infty^\ast$, and we have the equalities
    \begin{equation}\label{cl}L(E/\mathscr{O}_L)=\left[\operatorname{Lie}_E(\mathscr{O}_L):U(E;\mathscr{O}_L)\right]_{A}\left[H(E;\mathscr{O}_L)\right]_{A}=\left[\operatorname{Lie}_E(\mathscr{O}_L):U_{\operatorname{st}}(E;\mathscr{O}_L)\right]_{A}.\end{equation}
\end{theoreme}

\subsection{Evaluation at \texorpdfstring{$z=\zeta\in\overline{\F}_q$}{}.}\label{3.4}

We want to extend the notion of Stark units by evaluating the variable $z$ at $z=\zeta$ for all $\zeta\in \overline{\F}_q$.\\
Let $\zeta$ be an element of $ \overline{\F}_q$ and consider $\F_q(\zeta)$ the finite field obtained by adding $\zeta$ to $\F_q$. Let us define the ring $A_{\zeta}=\F_q(\zeta)\otimes_{\F_q} A$. We define a Frobenius $\tau_\zeta=\id\otimes \tau$ acting on $A_\zeta$.
Let us define $\widetilde{A_\zeta}=\F_q(z)\otimes_{\F_q}A_\zeta$ on which we extend the Frobenius $\tau_\zeta$ by $\tau_\zeta =\operatorname{id}\otimes \tau_\zeta$, still denoted by $\tau_\zeta$ (i.e., the Frobenius $\tau_\zeta$ acts as the identity on $\F_q(z)$).
Denote by $A_\zeta[z]=\F_q[z]\otimes_{\F_q}A_\zeta$. Set $\mathscr{O}_{L,\zeta}=\F_q(\zeta)\otimes_{\F_q} \mathscr{O}_L$. It is also equiped with the following Frobenius $\tau_\zeta=\id\otimes \tau$.

Similarly as the $z$-deformation, let us twist the $t$-module $E$ into an Anderson $A_\zeta$-module $E_\zeta$ defined over $M_d(\mathscr{O}_{L,\zeta})$ by 
$$(E_\zeta)_\theta=\Sum\limits_{i=0}^{r}E_{\theta,i}\zeta^i\tau_\zeta^i\in M_d(\mathscr{O}_{L,\zeta})\{\tau_\zeta\}$$ then extend to $A_\zeta$ by $\F_q(\zeta)$-linearity.

Set $M_w= \F_q(\zeta)\otimes_{\F_q} L_w$ where $w=\infty$ or $w=P$. Consider $\F_q[z]\otimes_{\F_q} \mathscr{O}_{L,\zeta}=\mathscr{O}_{L,\zeta}[z]$ then set $\widetilde{\mathscr{O}_{L,\zeta}}=\F_q(z)\otimes_{\F_q} \mathscr{O}_{L,\zeta}$ and $\widetilde{M}_{w}=\mathscr{O}_{L,\zeta}[z]\otimes_{\mathscr{O}_L[z]}\T_z(L_w)\simeq \F_q(\zeta)\otimes_{\F_q}\T_z(L_w)$ and consider $\F_q(\zeta)\otimes_{\F_q} \widetilde{L_w}$. We extend $v_\infty$ to $M_\infty$ as follows.
Let's fix $(f_1,\ldots, f_m)$ a $\F_q$-basis of $\F_q(\zeta)$. We set
$$v_\infty\left(\Sum\limits_{i=1}^{m}f_i\otimes x_i\right)=\min_{i=1,\ldots, h} v_\infty(x_i)$$ for $x_i\in K_\infty$. The topology over $M_\infty$ does not depend on the choice of the basis $(f_1,\ldots, f_m)$. We then consider $v_\infty$ an over-additive valuation on the $\F_q(\zeta)\otimes_{\F_q} K_\infty$-vector space of finite dimension $M_\infty$.  We then extend similarly $v_\infty$ to $\F_q(\zeta)\otimes_{\F_q} \widetilde{L_\infty}$. Remark that we cannot just replace $v_\infty$ by $v_P$ on $\F_q(\zeta)\otimes_{\F_q} K_P$ with theses constructions, in fact we do not obtain a valuation over $M_{v_P}$. See Subsection \ref{extensionP} for more details.

Finally, we deform $E$ into $E^{(\zeta)}$ an Anderson $A_\zeta$-module on $M_d(\mathscr{O}_{L,\zeta})$ by
$$E^{(\zeta)}_\theta=\Sum\limits_{i=0}^{r}E_{\theta,i}\tau_\zeta^i$$ and extend it to $A_\zeta$ by $\F_q(\zeta)$-linearity. We finally extend it to an Anderson $\widetilde{A_\zeta}$-module $\widetilde{E}^{(\zeta)}$ on $M_d(\widetilde{\mathscr{O}_{L,\zeta}})$ in the usual way.

We have exponential maps associated with each of the Anderson modules. From the definitions we have the equalities
    $$\exp_{E^{(\zeta)}}=\Sum\limits_{n\geq 0} d_n\tau_\zeta^n \text{ and }\log_{E^{(\zeta)}}=\Sum\limits_{n\geq 0} l_n\tau_\zeta^n,$$ 
and the map $\exp_{E^{(\zeta)}}$ (resp. $\exp_{\widetilde{E}^{(\zeta)}}$) converges on $\operatorname{Lie}_{E^{(\zeta)}}(M_\infty)$ (resp. on $\operatorname{Lie}_{\widetilde{E}^{(\zeta)}}(\widetilde{M}_{\infty})$).
Moreover, we have the following equalities in $M_d(\F_q(\zeta)\otimes_{\F_q}L)\{\{\tau_{\zeta}\}\}$:
$$\exp_{E_\zeta}=\Sum\limits_{n\geq 0} d_n \zeta^n\tau_\zeta^n \text{ and } \log_{E_\zeta}=\Sum\limits_{n\geq 0} l_n \zeta^n\tau_\zeta^n. $$

Consider the evaluation morphism at $z=\zeta$:
$$\operatorname{ev}_\zeta=\operatorname{ev}_{z=\zeta}:\widetilde{M}_{\infty}\rightarrow M_\infty$$ whose kernel is given by $(z-\zeta)\widetilde{M}_\infty$, then we consider
the following evaluation morphism still denoted by $\operatorname{ev}_{\zeta}$:
$$\operatorname{ev}_\zeta:\operatorname{Lie}_{\widetilde{E}^{(\zeta)}}(\widetilde{M}_{\infty})\rightarrow \operatorname{Lie}_{E^{(\zeta)}}(M_\infty).$$
For $x\in \Lie_{\widetilde{E}^{(\zeta)}}(\widetilde{M}_{\infty})$, we have in $\Lie_{E^{(\zeta)}}(M_\infty)$:
$$\operatorname{ev}_{\zeta}(\exp_{\widetilde{E}^{(\zeta)}}(x))=\exp_{E_\zeta} (\operatorname{ev}_\zeta (x)).$$

Let us consider the module of $\zeta$-units at the integral level:
$$U(\widetilde{E}^{(\zeta)}; \mathscr{O}_{L,\zeta}[z])=\left\{x\in \Lie_{\widetilde{E}^{(\zeta)}}(\widetilde{M}_{\infty}) \ens \exp_{\widetilde{E}^{(\zeta)}} (x)\in \widetilde{E}(\mathscr{O}_{L, \zeta}[z])\right\}$$ as well as the module of the $\zeta$-classes at the integral level:

$$H(\widetilde{E}^{(\zeta)}; \mathscr{O}_{L,\zeta}[z])=\dfrac{\widetilde{E}^{(\zeta)}(\widetilde{M}_{\infty})}{\widetilde{E}^{(\zeta)}(\mathscr{O}_{L,\zeta}[z])+\exp_{\widetilde{E}^{(\zeta)}}(\operatorname{Lie}_{\widetilde{E}^{(\zeta)}} (\widetilde{M}_{\infty}))}=A_\zeta[z]\otimes_{\F_q[z]}H(\widetilde{E},\mathscr{O}_L[z])$$
provided with a structure of $A_\zeta[z]$-modules.
 Next, consider the $\zeta$-unit module:
$$U(E_\zeta;\mathscr{O}_{L,\zeta})=\{x\in \Lie_{E_\zeta}(M_\infty) \ens \exp_{E_\zeta} (x)\in E_\zeta(\mathscr{O}_{L,\zeta})\}$$ and the $\zeta$-class module
$$H(E_\zeta,\mathscr{O}_{L,\zeta})=\dfrac{E_\zeta(M_\infty)}{E_\zeta(\mathscr{O}_{L,\zeta})+\exp_{E_\zeta} (\Lie_{E_\zeta}(M_\infty))}$$provided with their $A_\zeta$-module structure via $E_\zeta.$

Results to come in this section are adapted from \cite{Flo} and \cite{units}.

\begin{prop}
\begin{enumerate}
\item The exponential map $\exp_{E_\zeta}:\Lie_{E_\zeta}(M_{\infty})\rightarrow E_\zeta(M_\infty)$ is locally an isometry.
\item The exponential map $\exp_{\widetilde{E}^{(\zeta)}}:\Lie_{\widetilde{E^{(\zeta)}}}(\widetilde{M}_{\infty})\rightarrow \widetilde{E}^{(\zeta)}(\widetilde{M}_{\infty})$ is locally an isometry.
\end{enumerate}
\end{prop}

\begin{proof}
The proof is a direct corollary of Lemma \ref{lem:expinfini}, we omit the proof.
\end{proof}

\begin{prop}
\begin{enumerate}
    \item The module of $\zeta$-classes $H(E_\zeta; \mathscr{O}_{L,\zeta})$ is a $\F_q(\zeta)$-vector space of finite dimension, hence a torsion $A_\zeta$-module of finite type.

    \item The module of $\zeta$-units $U(E_\zeta;\mathscr{O}_{L,\zeta})$ is an $A_\zeta$-lattice in $M_\infty$.
    \item The class module $H(\widetilde{E}^{(\zeta)}; \mathscr{O}_{L,\zeta}[z])$ is a $\F_q(\zeta)[z]$-module of finite type.
    \end{enumerate}
\end{prop}
\begin{proof}
 The proof follows the proof of \cite[Proposition 2.6]{Flo} for the two first assertions, and the proof of \cite[Proposition 2]{units} for the last one, by replacing $A$ by $A_\zeta$, $\oo$ by $\mathscr{O}_{L,\zeta}$ and $E$ by $E_\zeta$. We omit the details.
\end{proof}

Just as Stark's units consist of the evaluation at $z=1$ of the $z$-units, we define the evaluation at $z=\zeta$ of the $\zeta$-units at the integral level: 
$$U_{\zeta}(E; \mathscr{O}_{L})=\operatorname{ev}_{\zeta}U(\widetilde{E}^{(\zeta)}; \mathscr{O}_{L,\zeta}[z])\subseteq U(E_\zeta; \mathscr{O}_{L,\zeta})$$
provided with an $A_\zeta$-module structure via $E_\zeta$.

\begin{theoreme}\label{Iso} There exists an $A_\zeta$-module isomorphism:
$$\dfrac{U(E_\zeta;\mathscr{O}_{L,\zeta})}{U_\zeta(E;\oo)}\simeq H(\widetilde{E}^{(\zeta)};\mathscr{O}_{L,\zeta}[z])[z-\zeta]$$
where $ H(\widetilde{E}^{(\zeta)};\mathscr{O}_{L,\zeta}[z])[z-\zeta]$ is the $(z-\zeta)$-torsion of the $\zeta$-class module at the integral level.
\end{theoreme}
 In the following, we will denote by $M=\mathscr{O}_{L,\zeta}$ and $\widetilde{M}=\mathscr{O}_{L,\zeta}[z]$.
\begin{proof} We follow the proof of \cite[Proposition 2.6]{units}. \\
    Consider the map
    $$\begin{aligned} \alpha:M_\infty^d&\rightarrow \widetilde{M}_{\infty}^d \\ x&\mapsto \dfrac{\exp_{\widetilde{E}^{(\zeta)}}(x)-\exp_{E_\zeta} (x)}{z-\zeta}.\end{aligned}$$
We divide the proof into several steps.\\
\underline{Step 1:} The map is well defined since 
    $$\operatorname{ev}_{\zeta} (\exp_{\widetilde{E}^{(\zeta)}} (x))=\exp_{E_\zeta} (x)$$ for $x\in M_\infty^d$, thus
    $(z-\zeta)$ divide $\exp_{\widetilde{E}^{(\zeta)}}(x)-\exp_{E_\zeta} (x)$ in $\widetilde{M}_{\infty}^d.$\\
\underline{Step 2:} We still denote $\alpha$ to be the restriction: $\alpha:U(E_\zeta,M)\rightarrow H(\widetilde{E}^{(\zeta)};\widetilde{M}).$ Let us prove that it is a homomorphism of $A_\zeta$-modules. Let $x\in U(E_\zeta;M)$ be a unit and $a\in A_\zeta$. Then:
    $$\begin{aligned} (z-\zeta)\alpha(ax)&=\exp_{\widetilde{E}^{(\zeta)}} (ax)-\exp_{E_\zeta} (ax) \\
    &=\widetilde{E}^{(\zeta)}_a (\exp_{\widetilde{E}^{(\zeta)}} (x))-(E_\zeta)_a(\exp_{E_\zeta} (x)) \\
    &=\Sum\limits_{i=0}^{r_a}E_{a,i}z^i\tau_\zeta^i(\exp_{\widetilde{E}^{(\zeta)}} (x))-\Sum\limits_{i=0}^{r_a}E_{a,i}\zeta^i\tau_\zeta^i(\exp_{E_\zeta}(x)) \\
    &=\Sum\limits_{i=0}^{r_a}E_{a,i}z^i\tau_\zeta^i(\exp_{\widetilde{E}^{(\zeta)}} (x)-\exp_{E_\zeta }(x))+\Sum\limits_{i=1}^{r_a}E_{a,i}(z^i-\zeta^i)\tau_\zeta^i(\exp_{E_\zeta}(x)).
    \end{aligned}$$
    Thus
    $$\alpha(ax)= \widetilde{E}_a^{(\zeta)}(\alpha(x))+\underbrace{\Sum\limits_{i=0}^{h}a_i\dfrac{z^i-\zeta^i}{z-\zeta}\tau_\zeta^i(\underbrace{\exp_{E_\zeta} (x)}_{\in M^d})}_{\in \widetilde{M}^d}.$$
    We have proved that $\alpha(ax)=\widetilde{E}_a^{(\zeta)}(\alpha(x))\text{ mod }\left(\widetilde{M}^d+\exp_{\widetilde{E}^{(
    \zeta)}} (\Lie_{\widetilde{E}^{(\zeta)}}(\widetilde{M}_{s,\infty}))\right),$ so $\alpha(ax)=\widetilde{E}_a^{(\zeta)} (\alpha(x))$ in $H(\widetilde{E}^{(\zeta)},\widetilde{M})$.\\
\underline{Step 3:} We claim that the image of $U(E_\zeta,M)$ is in the $(z-\zeta)$-torsion of the $\zeta$-class module at the integral level. In fact, let $x\in U(E_\zeta;M)$ be a unit. We have:
    $$(z-\zeta)\alpha(x)=\exp_{\widetilde{E}^{(\zeta)}}(x)-\exp_{E_\zeta} (x) =0 \text{ mod } \left(E_\zeta(M)+\exp_{\widetilde{E}^{(\zeta)}}(\Lie_{\widetilde{E}^{(\zeta)}} (\widetilde{M}_{\infty}))\right).$$
\underline{Step 4:} Let us prove the surjectivity of $\alpha$ on $H(\widetilde{E}^{(\zeta)};\widetilde{M})[z-\zeta]$. Let $x\in \widetilde{E}^{(\zeta)}(\widetilde{M}_{\infty})$ be such that
    $$(z-\zeta)x=\exp_{\widetilde{E}^{(\zeta)}}(u)+v$$ with $u\in \operatorname{Lie}_{\widetilde{E}^{(\zeta)}}(\widetilde{M}_{\infty})$ and $v\in \widetilde{E}^{(\zeta)}(\widetilde{M}).$ We write
    $u=u_1+(z-\zeta)u_2$, $u_1\in M_\infty^d$, $u_2\in \widetilde{M}_{\infty}^d$ and $v=v_1+(z-\zeta)v_2$, $v_1\in M^d$, $v_2\in \widetilde{M}^d$. We have:
    $$(z-\zeta)x=\exp_{\widetilde{E}^{(\zeta)}} (u_1)+v_1+(z-\zeta)(v_2+\exp_{\widetilde{E}^{(\zeta)}}(u_2)).$$
    By evaluating at $z=\zeta$ yields $\exp_{E_\zeta} (u_1) + v_1=0$. Thus $u_1\in U(E_\zeta;M).$ Moreover, we have:
    $$\alpha(u_1)=x-(\underbrace{\exp_{\widetilde{E}^{(\zeta)}} (u_2)}_{\in  \exp_{\widetilde{E}^{(\zeta)}}( \widetilde{M}_{\infty}^d)}+\underbrace{v_2}_{\in \widetilde{M}^d})$$ thus
    $\alpha(u_1)=x \text{ mod } \left(\widetilde{M}^d+\exp_{\widetilde{E}^{(\zeta)}} (\widetilde{M}_{\infty}^d)\right).$

\underline{Step 5:} We now consider the kernel of $\alpha:U(E_\zeta;M)\rightarrow H(\widetilde{E}^{(\zeta)},\widetilde{M})$ denoted by $\kappa$. We want to prove that $\kappa=U_{\zeta}(E,\oo).$ We proceed by double inclusion. \\
    \fbox{$\supseteq$} Let $x\in U_{\zeta}(E,\oo)$ be a unit and write
    $x=\operatorname{ev}_{\zeta} (u)$ with $u\in U(\widetilde{E}^{(\zeta)};\widetilde{M}).$ We have $\operatorname{ev}_{\zeta} (x-u)=0$ thus we can find $v\in \widetilde{M}_{\infty}^d$ such that
    $$x=u+(z-\zeta)v.$$ We have
    $$\alpha(x)=\dfrac{\exp_{\widetilde{E}^{(\zeta)}}(u) -\exp_{E_{\zeta}} (x)}{z-\zeta} + \exp_{\widetilde{E}^{(\zeta)}} (v).$$
    But $\exp_{E_\zeta}(x)=\operatorname{ev}_{\zeta} \exp_{\widetilde{E}^{(\zeta)}}(u) \in M^d$ so $\alpha(x) =0 \text{ mod } \left(\widetilde{M}^d+\exp_{\widetilde{E}^{(\zeta)}} (\widetilde{M}_{\infty}^d)\right).$ \\
    \fbox{$\subseteq$} Let $x\in U(E_\zeta;M)$ be such that $\alpha(x)\in \widetilde{M}^d+\exp_{\widetilde{E}^{(\zeta)}} (\widetilde{M}_{\infty}^d)$. Let us express $\alpha(x)=u+\exp_{\widetilde{E}^{(\zeta)}} (v)$. We have
    $$(z-\zeta)\alpha(x)=\exp_{\widetilde{E}^{(\zeta)}} (x)+\exp_{E_\zeta} (x)=(z-\zeta)u+\exp_{\widetilde{E}^{(\zeta)}}((z-\zeta)v).$$
Thus $\exp_{\widetilde{E}^{(\zeta)}} (x-(z-\zeta)v)=(z-\zeta)u +\exp_{E_\zeta} (x) \in \widetilde{M}^d$ so $x-(z-\zeta)v\in U(\widetilde{E}^{(\zeta)}; \widetilde{M})$. Finally we obtain
$$\operatorname{ev}_{\zeta} (x-(z-\zeta)v)=x \in U_\zeta(E;\oo).$$

\end{proof}

\begin{cor}\label{lattice} The unit module $U_\zeta(E;\oo)$ is an $A_\zeta$-lattice in $M_\infty^d$. Moreover, we have the following equalities
$$\left[H(\widetilde{E}^{(\zeta)}; \widetilde{M})[z-\zeta]\right]_{A_\zeta}=\left[H(E_\zeta;M)\right]_{A_\zeta}=\left[\dfrac{U(E_\zeta;M)}{U_\zeta(E;\oo)}\right]_{A_\zeta}.$$
\end{cor}
Let us start by proving the following result.

\begin{lem}\label{exacte} We have an exact sequence of $A_\zeta$-modules:

$$0\longrightarrow (z-\zeta)H(\widetilde{E}^{(\zeta)}; \widetilde{M})\longrightarrow H(\widetilde{E}^{(\zeta)};\widetilde{M})\overset{\operatorname{ev}_{\zeta}}{\longrightarrow}H(E_\zeta;M)\longrightarrow 0.$$
\end{lem}
\begin{proof}[Proof of Lemma \ref{exacte}]
    We apply the snake lemma to the following commutative diagram (where the lines are exact sequences of $A_\zeta$-modules)
and the $i_j$ represent natural injections:
    \begin{center}
\begin{tikzpicture}[xscale=3.5,yscale=3]


\node (B) at (0,0) {$(z-\zeta)(\exp_{\widetilde{E}^{(\zeta)}}(\widetilde{M}_{\infty}^d)+\widetilde{M}^d)$};
\node (C) at (1.35,0) {$\exp_{\widetilde{E}^{(\zeta)}}(\widetilde{M}_{\infty}^d)+\widetilde{M}^d$};
\node (D) at (2.55,0) {$\exp_{E_\zeta}(M_\infty^d)+M^d$};

\node (G) at (0,-0.6) {$(z-\zeta)\widetilde{M}_{\infty}^d$};
\node (H) at (1.35,-0.6) {$ \widetilde{M}_{\infty}^d$};
\node (I) at (2.55,-0.6) {$M_\infty^d$};


\draw[->,thick] (B) -- (C);
\draw[->,thick] (C) -- (D) node[pos=0.5,above]{$\operatorname{ev}_{z=\zeta}$} ;

\draw[->,thick] (G) -- (H);
\draw[->,thick] (H) -- (I) node[pos=0.5,above] {$\operatorname{ev}_{z=\zeta}$};

\draw[->,thick] (B) -- (G) node[pos=0.5,right] {$i_1$};
\draw[->,thick] (C) -- (H) node[pos=0.5,right] {$i_2$};
\draw[->,thick] (D) -- (I) node[pos=0.5,right] {$i_3$} ;

\end{tikzpicture}
\end{center}

\end{proof}

\begin{proof}[Proof of Corollary \ref{lattice}]
We deduce by Lemma \ref{exacte} an exact sequence of $\F_q(\zeta)$-vector spaces of finite dimension (and of finitely-generated $A_\zeta$-modules):

$$0 \longrightarrow H(\widetilde{E}^{(\zeta)};\widetilde{M})[z-\zeta]\longrightarrow H(\widetilde{E}^{(\zeta)};\widetilde{M})\overset{.(z-\zeta)}{\longrightarrow} H(\widetilde{E}^{(\zeta)},\widetilde{M})\overset{\operatorname{ev}_{\zeta}}{\longrightarrow}H(E_\zeta;M)\longrightarrow 0.$$
By Proposition \ref{fitt} we obtain:

$$\left[H(\widetilde{E}^{(\zeta)}; \widetilde{M})[z-\zeta]\right]_{A_\zeta}=\left[H(E_\zeta;M)\right]_{A_\zeta}=\left[\dfrac{U(E_\zeta;M)}{U_\zeta(E;\oo)}\right]_{A_\zeta},$$
the last equality coming from Theorem \ref{Iso}. Since $H(\widetilde{E}^{(\zeta)};\widetilde{M})[z-\zeta]$ is a $\F_q(\zeta)$-vector space of finite dimension and $U(E_\zeta;M)$ is an $A_\zeta$-lattice in $M_\infty^d$, the result follows.
    \end{proof}

\section{The \texorpdfstring{$P$}{TEXT}-adic case}\label{section:padique}

We keep the notation of Section \ref{section:infini}. Recall that $L$ is a finite extension of $K$ of degree $n$, $\oo$ denotes the integral closure of $A$ in $L$ and $E$ is an Anderson $t$-module defined over $\oo$ of dimension $d$. The goal of this section is to define and study some $P$-adic $L$-series associated with Anderson $t$-modules by removing the local factor at $P$ of the classical $L$-series.

\subsection{Introduction and notation}
Recall that the local factor at $Q$ associated with $E$ is defined by 
$z_Q(\widetilde{E}/\widetilde{\oo})=\frac{\left[\operatorname{Lie}_{\widetilde{E}}(\widetilde{\oo}/Q\widetilde{\oo})\right]_{\widetilde{A}}}{\left[\widetilde{E}(\widetilde{\oo}/Q\widetilde{\oo})\right]_{\widetilde{A}}}\in \widetilde{K}$ (resp. $z_Q(E/\oo)=\frac{\left[\operatorname{Lie}_E(\mathscr{O}_L/Q\mathscr{O}_L)\right]_{A}}{\left[E(\mathscr{O}_L/Q\mathscr{O}_L)\right]_{A}}\in K$). The goal of this section is to study the following infinite product of local factors $z_Q(E/\oo)$ (resp. $z_Q(\widetilde{E}/\widetilde{\oo})$) that we call the $P$-adic $L$-series (resp. the $z$-twisted $P$-adic $L$-series):

$$L_P(\widetilde{E}/\widetilde{\oo})=\Prod\limits_{Q\neq P} z_Q(\widetilde{E}/\widetilde{\oo}) \text{ (resp. } L_P({E}/{\mathscr{O}}_{L})=\Prod\limits_{Q\neq P} z_Q({E}/{\mathscr{O}}_{L}))$$ where $Q$ runs through the monic primes of $A$ different from $P$. 
   
More precisely, let us denote by $v_P$ a finite place of $K$ associated to an irreducible monic polynomial $P$ of $A$. Let $\F_P=\F_{q^{\deg(P)}}$ be the residue field associated to $v_P$ and let $K_P=\F_P((P))$ (resp. $A_P=\F_P[[P]]$) be the completion of $K$ (resp. $A$ )for $v_P$. Let $\C_P$ be the completion of an algebraic closure of $K_P$ and $v_P$ the valuation on $\C_P$ normalized such that $v_P(P)=1$. Set $\widetilde{K_P}=\F_q(z)\hat{\otimes}_{\F_q(z)}K_P=\F_P(z)((P))$ on which we extend the valuation $v_P$:
$$v_P(\Sum\limits_{n\geq N} \alpha_n(z)P^n)=N, N\in \Z, \alpha_n(z)\in \F_P(z), \alpha_N(z)\neq 0.$$
Let $\vert\cdot \vert_P$ be the absolute value on $\C_P$ defined by $\vert x\vert_P=q^{-v_P(x)}.$ Let $\mathscr{B}=(f_1,\ldots, f_n)$ be an $A$-basis of $\oo$ (that is also a $K$-basis of $L$). We set $L_{P}=L\otimes_{K} K_P$ and $\widetilde{L_P}=L\otimes_{K}\widetilde{K_P}$. In what follows, the reader will be careful not to confuse the notation $L_P(E/\oo)$ for the $P$-adic $L$-series and $L_P$ for the tensor product $L\otimes_K K_P$. Then $L_{P}$ is a $K_P$-vector space with $\mathscr{B}$ as a basis, and $\widetilde{L_P}$ is a $\widetilde{K_P}$-vector space with $\mathscr{B}$ as a basis. In particular on $L_P$ all the norms of $K_P$-vector space of finite dimension are equivalent. Let us work with the following.

 Consider the sup norm $\vert.\vert_P$ with respect to this basis. In other words, if $x= \Sum\limits_{i=1}^n f_i\otimes x_i$ with the $x_i\in K_P$, then we set $$\vert x\vert_P = \max_{i=1,\ldots, n}\vert x_i \vert_P.$$
We obtain over $L_P$ a norm of $K_P$-algebra. We then consider the over-additive valuation of $K_P$-vector spaces of finite dimension on $L_P$ defined by:
$$v_P(x)=-\log_q \vert x\vert_P=\min_{i=1,\ldots,d} v_P(x_i).$$
For all $d\geq 1$, we extend these definitions to $L_P^d$: if $x=(x_1,\ldots, x_d)\in L_P^d$, then we set
$$\vert x\vert=\max_{i=1,\ldots, d} \vert x_i\vert$$ or equivalently
$$v_P(x)=-\log_q\vert x\vert =\min_{i=1,\ldots,d} v_P(x_i).$$ 
In particular for all $x\in \oo^d$ we get $v_P(x)\geq 0$. Consider $$\T_z(K_P)=\left\{f(z)=\Sum\limits_{n\geq 0} a_nz^n \ens a_n\in K_P \text{ and } \lim\limits_{n\rightarrow +\infty} v_P(a_n)=+\infty\right\}\subset \widetilde{K_P}$$ and
$$\T_z(L_P)=\left\{f(z)=\Sum\limits_{n\geq 0} a_nz^n \ens a_n\in L_{P} \text{ and } \lim\limits_{n\rightarrow +\infty} v_P(a_n)=+\infty\right\}=L\otimes_{K}\T_z(K_P).$$
We define a $K_P$-vector space structure over $\operatorname{Lie}_E(L_P)$. We take inspiration from the $\infty$-adic case in \cite[Lemma 1.7]{Fang} and \cite[section 2.3]{Flo}.

\begin{prop}We can extend the homomorphism $\partial_E:A\rightarrow M_d(\oo)$ into a homomorphism from $K_P$ to $M_d(L_P)$ in the following way:
$$\begin{aligned} \partial_E: K_P&\rightarrow M_d(L_P), \\
\Sum\limits_{i\geq -N} \alpha_iP^i&\mapsto \Sum\limits_{i\geq -N}\alpha_i \partial_E(P)^{i}.\end{aligned}$$
Moreover, with respect to this action, $L_P^d$ is a $K_P$-vector space of dimension $m=d[L:K]$ denoted by $\operatorname{Lie}_E(L_P)$.
\end{prop}

\begin{proof}
Consider $n\geq 0$. There exists a unique integer $t_n$ such that 
$$q^dt_n\leq n <(t_n+1)q^d.$$
We have:
$$\partial_E(P)^{n}=\partial_E(P)^{q^dt_n}\partial_E(P)^{n-q^dt_n}=P^{q^dt_n}\underbrace{\partial_E(P)^{n-q^dt_n}}_{\in M_d(\oo)}.$$ We obtain that
$\lim\limits_{n\rightarrow +\infty}v_P(\partial_E(P)^n)=+\infty$ thus the map $\partial_E$ is well-defined. Denote by $W_P=\F_{q^{\deg (P)}}((P^{q^d}))\subseteq K_P.$ Then for all $x\in W_P$ we have $\partial_E(x)=xI_d$,thus we have an isomorphism of $W_P$-vector spaces:
$$\operatorname{Lie}_E(L_P)\simeq L_P^d.$$
Hence $\operatorname{Lie}_E(L_P)$ is a $K_P$-vector space of finite dimension. We have:
$$\dim_{W_P}(\operatorname{Lie}_E(L_P))=\dim_{K_P}(\operatorname{Lie}_E(L_P)) \dim_{W_P}(K_P).$$
But from the isomorphism of $W_P$-vector spaces $\operatorname{Lie}_E(L_P)\simeq L_P^d$ we have:
$$\dim_{W_P}(\operatorname{Lie}_E(L_P))=\dim_{W_P}(L_P^d)=\dim_{K_P}(L_P^d) \dim_{W_P}(K_P)$$
thus
$$\dim_{K_P}(\operatorname{Lie}_E(L_P))=\dim_{K_P}(L_P^d)=d[L:K].$$

\end{proof}
We also have that $\operatorname{Lie}_E(\oo)$ is an $A$-lattice in $\operatorname{Lie}_E(L_P)$. Finally, everything is still valid by adding the variable $z$, in other words $\operatorname{Lie}_E(\widetilde{L_P})$ is a $\widetilde{K_P}$-vector space of dimension $d[L: K]$ and $\operatorname{Lie}_{\widetilde{E}}(\widetilde{\oo})$ is an $\widetilde{A}$-lattice in $\operatorname{Lie}_E(\widetilde{L_P})$.
In particular, we have:
$$\partial_{\widetilde{E}}:\T_z(K_P))\rightarrow M_d(\T_z(L_P)).$$
Remark that the topologies of $L_P^d$ and $\operatorname{Lie}_E(L_P)$ are equivalent.

Consider the unique $t$-module $F$ over $\oo$ satisfying $P F_a=E_aP$ for all $a\in A$. If $E_a=\Sum\limits_{i=0}^{r_a}E_{a,i}\tau^i$, then $F_a=\Sum\limits_{i=0}^{r_a}E_{a,i}P^{q^i-1}\tau^i$. In particular for all $a\in A$ we have: $\partial_E(a)=\partial_F(a)$. From \cite[Section 3.2]{Admissible} we have the following equalities in $M_d(L)\{\{\tau\}\}$:
$$\log_{F}=P^{-1}\log_{E}P=\Sum\limits_{n\geq 0} l_nP^{q^n-1}\tau^n,$$ and
$$\exp_{F}=P^{-1}\exp_{E}P=\Sum\limits_{n\geq 0} d_nP^{q^n-1}\tau^n.$$

We now study the link between the local factors of $E$ and $F$.
\begin{lem}\label{tordulocal}
Let $Q\in A$ be an monic prime. If $Q\neq P$, then we have the following equalities: $z_Q({F}/{\mathscr{O}}_{L})=z_Q({E}/{\mathscr{O}}_{L})$ and $z_Q(\widetilde{F}/\widetilde{\oo})=z_Q(\widetilde{E}/\widetilde{\oo})$.
Otherwise $z_P(F/{\mathscr{O}_L})=1$ and $z_P(\widetilde{F}/\widetilde{\mathscr{O}_L})=1$.
\end{lem}
\begin{proof}
    See \cite[Lemma 3.7]{Admissible}.
\end{proof}
We then obtain the following result.

\begin{cor}
    We have the following equality in $K[[z]]$:
    $$L(\widetilde{F}/\widetilde{\oo})=L_P(\widetilde{E}/\widetilde{\oo}).$$ 
\end{cor}

\subsection{ \texorpdfstring{$P$}{}-adic exponential and \texorpdfstring{$P$}{}-adic logarithm}

We define $(D_n)_{n\geq 0}$ and $(L_n)_{n\geq 0}$ as the following sequences of elements of $A$:

$$\left\{\begin{aligned} &D_0=1, \\ &D_n=\prod\limits_{k=0}^{n-1}(\theta^{q^{n-k}}-\theta)^{q^k}, \end{aligned}\right. \text{ and } \left\{\begin{aligned} &L_0=1, \\ &L_n=\prod\limits_{k=1}^{n}(\theta-\theta^{q^k}).\end{aligned}\right .$$

We first estimate the $P$-adic valuation of $D_n$ and $L_n$ for all $n\geq 0$.
\begin{lem}\label{l1}
We have the following equalities for $n\geq 1$:
\begin{enumerate}
\item $ v_P(D_n)=q^n\dfrac{q^{-\deg (P)\left(\floor{\frac{n}{\deg (P)}}+1\right)}-q^{-\deg (P)}}{q^{-\deg (P)}-1}=q^n\dfrac{q^{-\deg (P)\floor{\frac{n}{\deg (P)}}}-1}{1-q^{\deg (P)}},$
\item $v_P(L_n)=\floor{\dfrac{n}{\deg (P)}}.$
\end{enumerate}
\end{lem}

\begin{proof} See \cite[Section 2]{a-t}.
\end{proof}

We recall that $\partial_E(a)\in M_d(A)$ is the constant coefficient of $E_a\in M_d(A)\{\tau\}$ for all $a\in A$, see Section \ref{section:notation}. Set $s\in \N$ the smallest integer such that $(\partial_E(\theta)-\theta I_d)^{q^s}=0.$ There exists because $\partial_E(\theta)-\theta I_d$ is nilpotent. Then for all $a\in A$ we have $\partial_E(a^{q^s})=a^{q^s}I_d$.\\
Recall that $\exp_E=\Sum\limits_{n\geq 0} d_n\tau^n\in M_d(L)\{\{\tau\}\}$ and $\log_E=\Sum\limits_{n\geq 0} l_n\tau^n\in M_d(L)\{\{\tau\}\}$.
Following \cite[Theorem 4.6.9]{Goss}, using functional equation of the logarithm map (resp. the exponential map) $\log_E E_{\theta^{q^s}}=\partial_E(\theta^{q^s})\log_E$, an immediate induction tells us that $l_n$ has the form
$$l_n=\dfrac{a_n}{L_n^{q^s}}$$ with $a_n\in M_d(\oo)$.

Reasoning in a similar way for the exponential map and by Lemma \ref{l1} we obtain the following result.

\begin{prop}\label{T1}We have the following inequalities for all $n\geq 0$:
\begin{enumerate}
\item $v_P(l_n)\geq -q^{s}\floor{\dfrac{n}{\deg (P)}},$
\item $v_P(d_n)\geq -q^{s+n}\dfrac{q^{-\deg (P)\floor{\frac{n}{\deg (P)}}}-1}{1-q^{\deg (P)}}.$
\end{enumerate}
\end{prop}

So far, we have considered the exponential and logarithm series as functions of $L_\infty^d$, but now we want to look at them as functions of $L_P^d$, which we denote by $\exp_{E,P}$ and $\log_{E,P}$. Note that formally (i.e., in $M_d(L)\{\{\tau\}\})$, these are always the same series. We do the same for $z$-twist.

Let us denote by $\operatorname{ev}_{z=1,P}:\T_z(L_P)^d\rightarrow L_P^d$ the $P$-adic evaluation morphism at $z=1$, whose kernel is given by $(z-1)\T_z(L_P)^d$.

We can first study the $P$-adic convergence domain of the $P$-adic logarithms maps associated with $\widetilde{F}$ and $\widetilde{E}$. We consider the following sets:\\
$\bullet$ $\Omega_z=\left\{x\in \T_z(L_{P})^d \ens v_P(x)\geq 0\right\}$ and $\Omega_z^+=\left\{x\in \T_z(L_{P})^d \ens v_P(x)> 0\right\}$,\\
$\bullet$ $\Omega=\left\{x\in L_{P}^d \ens v_P(x)\geq 0\right\}$ and $\Omega^+=\left\{x\in L_{P}^d \ens v_P(x)> 0\right\},$\\
$\bullet$ $\mathcal{D}_z=\left\{x\in \T_z(L_P)^d\ens v_P(x)>-1+\dfrac{q^s}{q^{\deg (P)}-1}\right\}$,\\
$\bullet$ $\mathcal{D}_z^+=\left\{x\in \T_z(L_P)^d\ens v_P(x)>\dfrac{q^s}{q^{\deg (P)}-1}\right\},$\\ 
$\bullet$ $\mathcal{D}=\left\{u\in (L_P)^d\ens v_P(u)>-1+\dfrac{q^s}{q^{\deg (P)}-1}\right\}$,\\
$\bullet$ $\mathcal{D}^+=\left\{u\in (L_P)^d\ens v_P(u)>\dfrac{q^s}{q^{\deg (P)}-1}\right\}.$
\begin{prop}\label{cvlogp}
\begin{enumerate}
\item   We have the $P$-adic convergences:
    $$     \log_{\widetilde{E},P}:\Omega_z^+\rightarrow \operatorname{Lie}_{\widetilde{E}}(\T_z(L_P)) $$ and
    $$ \log_{{E},P}:\Omega^+\rightarrow \operatorname{Lie}_{{E}}((L_P). $$
Moreover, $\log_{\widetilde{E},P}:\mathcal{D}_z^+\rightarrow \mathcal{D}_z^+$ is an isometry and $\log_{{E},P}:\mathcal{D}^+\rightarrow \mathcal{D}^+$ is an isometry.

     \item The first assertion remains true by replacing $E$ by $F$ and deleting the "$+$".
     As a particular case, we have the convergence
    $$\log_{\widetilde{F},P}:\widetilde{E}(\mathscr{O}_L[z])\rightarrow \operatorname{Lie}_{\widetilde{E}}(\T_z(L_P)).$$
    \end{enumerate}
    \end{prop}

       \begin{proof}
We give the proof only for $\log_{\widetilde{E},P}$, the arguments are similar in the other cases. Consider $f(z)\in \T_z(L_P)^d.$ We have (first formally):

$$\log_{\widetilde{E},P} f(z)=\Sum\limits_{n\geq 0} L_nz^n\tau^n(f(z)).$$ 
For all $n\geq 0$ we have:
$$v_P\left(L_n\tau^n(f(z))\right)\geq v_P(L_n)+v_P\left(\tau^n(f(z))\right)\geq -q^s\floor{\dfrac{n}{\deg (P)}}+ q^n v_P(f(z))$$ and this last quantity tends to $\infty$ when $n$ tends to $\infty$ if $v(f(z))>0$. Moreover, if $v_P(f(z))>\dfrac{q^s}{q^{\deg (P)}-1}$, then we have for all $n\geq 1$:
$$v_P\left(L_n\tau^n(f(z))\right)-v_P(f(z))\geq (q^n-1)v_P(f(z))+v(L_n)>\dfrac{q^s}{q^{\deg (P)}-1}(q^n-1)-q^s\floor{\dfrac{n}{\deg (P)}}.$$
 Write $n=b \deg (P)+i\geq 1$ with $b\in \N$ and $0\leq i<\deg (P)$. Then:
        $$\dfrac{q^s}{q^{\deg (P)}-1}(q^n-1)-q^s\floor{\dfrac{n}{\deg (P)}}=q^s\left(\dfrac{q^{b\deg (P)+i}-1}{q^{\deg (P)}-1}-b\right).$$
        But we have:
        $$\dfrac{q^{b\deg (P)+i}-1}{q^{\deg (P)}-1}-b\geq \dfrac{q^{b\deg (P)}-1}{q^{\deg (P)}-1}-b={1+q^{\deg (P)}+\ldots+(q^{\deg (P)})^{b-1}}-b\geq 0.$$
\end{proof}

We have results for the $P$-adic convergences of the exponentials series using similar arguments.

\begin{prop}\label{evexpp}
\begin{enumerate}
\item   We have the $P$-adic convergences:
    $$     \exp_{\widetilde{E},P}:\mathcal{D}_z^+\rightarrow \T_z(L_P)^d $$ and
    $$     \exp_{{E},P}:\mathcal{D}^+\rightarrow L_P^d .$$
Moreover, $\exp_{\widetilde{E},P}:\mathcal{D}_z^+\rightarrow \mathcal{D}_z^+$ is an isometry and $\exp_{{E},P}:\mathcal{D}^+\rightarrow \mathcal{D}^+$ is an isometry.

     \item The first assertion remains true by replacing $E$ by $F$ and deleting "$+$".
    \end{enumerate}
    \end{prop}

In particular for all $x\in \mathcal{D}_z $ we have the following $P$-adic equality:
\begin{equation}\label{egv}\exp_{F,P} (\operatorname{ev}_{z=1,P} (x))=\operatorname{ev}_{z=1,P} (\exp_{\widetilde{F}} (x)).\end{equation}
Similarly for all $x\in \Omega_z$ we have the following $P$-adic equality:
\begin{equation}\label{lgv}\log_{F,P}(\operatorname{ev}_{z=1,P}(x))=\operatorname{ev}_{z=1,P}(\log_{\widetilde{F}} (x)).\end{equation}
Similarly in their convergence domain, all of the $P$-adic exponential and logarithm maps verify the functional identities of the exponential and the logarithm maps: 

$$\begin{aligned}\forall (a,x)\in A\times\Omega_z, \ \log_{\widetilde{F},P} \partial_{\widetilde{F}}(a) x=\widetilde{F}_a\log_{\widetilde{F},P} x, \\
\forall (a,x) \in A\times\mathcal{D}_z, \  \exp_{\widetilde{F},P} \partial_{\widetilde{F}}(a) x=\widetilde{E}_a\log_{\widetilde{F},P} x.
\end{aligned} $$
Moreover, for all $x\in \mathcal{D}_z^+$ we have 
$$\exp_{\widetilde{E},P} (\log_{\widetilde{E},P} (x))=\log_{\widetilde{E},P} (\exp_{\widetilde{E},P}(x))=x.$$
The same goes without the variable $z$, and the same goes for $\widetilde{E}$ (resp. $E$) over $\Omega_z^+$ and $\mathcal{D}_z^+$ (resp. over $\Omega^+$ and $\mathcal{D}^+$).

\subsection{Evaluation at \texorpdfstring{$z=\zeta\in\overline{\F}_q$}{}: the \texorpdfstring{$P$}{}-adic setting}\label{extensionP}

Consider $\F_P$ to be the residual field associated with $P$. Set $\F=\F_q(\zeta)\cap \F_P$ and $G=\operatorname{Gal}(\F/\F_q)$. Let us first remark that the valuation $w$ defined in \ref{3.4} is not a valuation over $\F_q(\zeta)\otimes_{\F_q} K_P$.\\
We have an isomorphism of $\F$-vector spaces:
    $$\F\otimes_{\F_q} \F\simeq\Prod\limits_{g\in G} \F \simeq \Prod\limits_{g\in G} \left(\F\otimes_{\F}\F\right)$$ 
    given by
    $$\eta:x\otimes y\mapsto\left( g(x)y, g\in G\right). $$
In particular, through this isomorphism, the Frobenius $\tau_\zeta$ is identified with $(\id \otimes \tau,\ldots,\id\otimes \tau)$.\\
First, we extend the scalars from $\F$ to $\F_q(\zeta)$. We obtain a (canonical) isomorphism (of $\F_q(\zeta)$-vector spaces) $\eta':\F_q(\zeta)\otimes_{\F_q}\F\rightarrow \prod\limits_{g\in G}\F_q(\zeta)\otimes_\F \F$, given by the following.
Let $(f_1,\ldots, f_l)$ be an $\F$-basis of $\F_q(\zeta)$ and $a_1,\ldots, a_l\in \F$. We set
$$\eta'\left(\Sum\limits_{i=1}^{l}a_if_i\otimes_{\F_q} x_i\right)=\left(\Sum\limits_{i=1}^l g(a_i)f_i\otimes_{\F}x_i, g\in G \right).$$
Note that the isomorphism is canonical, but not the topologies that will appear. We then naturally extend the scalars (on the right) from $\F$ to $L_P$. We obtain an isomorphism (of $\F_q(\zeta)$-vector spaces on the left and $L_P$-modules on the right) induced by $\eta$, also denoted $\eta$:
    $$\eta:\F_q(\zeta)\otimes_{\F_q} L_P\simeq \Prod\limits_{g\in G} \left(\F_q(\zeta)\otimes_{\F}L_P\right).$$ 
In particular we obtain $L_P$-vector spaces of dimension $[\F_q(\zeta):\F]$ on each component of the product, so an $L_P$-vector space of dimension $[\F_q(\zeta):\F_q]$. For $g\in G$, we set $H_g=\F_q(\zeta)\otimes_{\F} K_P\simeq \F_q(\zeta)((P))$ (where the action on the left of $\F$ is determined by $g\in G$).

For $x=\Sum\limits_{i=1}^l f_i\otimes x_i\in H_g$, we consider the usual valuation on $H_g$:
$$v_g(x)=\min_{i=1,\ldots,m} v_P(x_i).$$

For all $g\in G$ we provide $L\otimes_K H_g$ with the topology $v_g$ induced by its structure of $H_g$-vector space of finite dimension with respect to the choice of the basis $\mathscr{B}$ of $L/K$. In particular, if we set $\operatorname{pr}_g$ the projection on the $g$-component of the product, then we obtain $v_g\left(\operatorname{pr}_g\left( \eta\left((\F_q(\zeta)\otimes_{\F} \oo)\right)\right)\right)\geq 0$.

Let $v_P$ be the over-additive valuation on the product $\Prod_{g\in G} (L\otimes_K H_g)$:
$$v_P(x_g,g\in G)=\min_{g\in G} (v_g(x_g))$$ verifying $v_P(\eta(1\otimes P))=1$.
Remark that the Frobenius $\tau_\zeta$ is equal to $(\id\otimes \tau,\ldots, \id\otimes \tau)$ on $\Prod\limits_{g\in G} \left(\F_q(\zeta)\otimes_{\F} L_P\right)$.

\begin{rem} Following exactly the same ideas, by extending the scalars from $\F$ to $\T_z(L_P)$ or $\widetilde{L_P}$ we obtain the isomorphisms of $\F_q(\zeta)$-vector spaces on the left and $\widetilde{L_P}$ (resp. $\T_z(L_P)$) on the right:

    $$\eta_z:\F_q(\zeta)\otimes_{\F_q}\T_z(L_P)\simeq \Prod\limits_{g\in G}\F_q(\zeta)\otimes_{\F}\T_z(L_P)$$
    and
     $$\F_q(\zeta)\otimes_{\F_q}\widetilde{L_P}\simeq \Prod\limits_{g\in G}\F_q(\zeta)\otimes_{\F}\widetilde{L_P}.$$
\end{rem}
We are now interested in the case of the higher dimension $d$.
We extend $v_P$ onto $(\F_q(\zeta)\otimes_{\F_q} L_P)^d$ (the same goes with $z$) (topology of finite-dimensional vector spaces, for example with respect to the canonical basis).
Then set $$\Omega_{\zeta,d}=\{x\in (\F_q(\zeta)\otimes_{\F_q} L_P)^d \ens v_P(x)\geq 0\}\supseteq (\F_q(\zeta)\otimes \oo)^d$$ and 

$$\Omega_{\zeta,d,z}=\{x\in (\F_q(\zeta)\otimes_{\F_q} \T_z(L_P))^d \ens v_P(x)\geq 0\}\supseteq (\F_q(\zeta)\otimes_{\F_q}\oo[z])^d.$$

\begin{prop}\label{logpzeta}
    We have the following convergences:
    $$\log_{F^{(\zeta)},P}:\Omega_{\zeta,d}\rightarrow (\F_q(\zeta)\otimes_{\F_q} L_P)^d$$
    and 
     $$\log_{F^{(\zeta)},P}:\Omega_{\zeta,d,z}\rightarrow (\F_q(\zeta)\otimes_{\F_q} \T_z(L_P))^d.$$
\end{prop}

\begin{proof} It follows from Proposition \ref{T1} and the definitions of the objects. We omit the proof.
\end{proof}

\subsection{The \texorpdfstring{$P$}{}-adic \texorpdfstring{$L$}{}-series}

Recall that $m=d[L:K]$ where $d$ in the dimension of the $t$-module $E$ and that $F$ is the $t$-module given by $F=P^{-1}EP$. Let $\mathscr{C}=(g_1,\ldots,g_m)$ be an $A$-basis of $\operatorname{Lie}_E(\oo)$, it is also a $\widetilde{K_\infty}$-basis of $\operatorname{Lie}_{\widetilde{E}}(L_\infty)$, a $\widetilde{K_P}$-basis of $\operatorname{Lie}_{\widetilde{E}}(L_P)$ and a $\T_z(K_P)$-basis of $\operatorname{Lie}_{\widetilde{E}}(\T_z(L_P))$. The same goes by replacing $E$ by $F$ since $\partial_E=\partial_F$. 

Let us remark, from Corollary \ref{cvlogp}, that for any $z$-unit $y(z)\in U(\widetilde{F},\mathscr{O}_L[z])$ we have $\exp_{\widetilde{F}}(y(z))\in \widetilde{E}(\mathscr{O}_L[z])\subseteq \Omega_z$ and therefore
$$\log_{\widetilde{F},P}( \exp_{\widetilde{F}} (y(z) ))\in \operatorname{Lie}_{\widetilde{E}}(\T_z(L_P)).$$
Moreover, for a family $(x_1(z),\ldots, x_m(z))$ of elements of $\operatorname{Lie}_{\widetilde{E}}(\T_z(L_P))$ we have
$$\operatorname{Mat}_{\mathscr{C}}(x_1(z),\ldots,x_m(z))\in M_m(\T_z(K_P))$$
thus
$$\operatorname{det}_{\mathscr{C}}(x_1(z),\ldots,x_m(z))\in \T_z(K_P).$$

Next, formally in $(L[[z]])^d$ we have the following equality for all $f(z)\in (L[[z]])^d$:

$$\log_{\widetilde{F},P}( \exp_{\widetilde{F}} (f(z)))=f(z).$$
Let $\left(v_1(z),\ldots,v_m(z)\right)\subset U(\widetilde{F};\mathscr{O}_L[z])$ be an $\widetilde{A}$-basis of $U(\widetilde{F};\widetilde{\oo}).$ Remark that the family \\$(1\otimes v_1(z),\ldots, 1\otimes v_m(z))\subseteq \F_q(\zeta)\otimes_{\F_q}\T_z(L_\infty)^d$ is also an $\widetilde{A_\zeta}$-basis of $U(\widetilde{F}^{(\zeta)};\widetilde{M})$. Set
$$w(z)=\det_{\mathscr{C}}\left(v_1(z),\ldots,v_m(z)\right)\in \T_z(K_\infty)$$ and
$$w_P(z)=\det_{\mathscr{C}}\left(\log_{\widetilde{F},P} (\exp_{\widetilde{F}}(v_1(z))),\ldots,\log_{\widetilde{F},P} (\exp_{\widetilde{F}}(v_m(z)))\right)\in \T_z(K_P).$$
By the above discussions and the class formula, we have the following equality in $K[[z]]$: 
$$L_P(\widetilde{F}/\widetilde{\oo})=\dfrac{w_P(z)}{\operatorname{sgn}(w(z))}.$$
Since $w_P(z)\in \T_z(K_P)$, to study the $P$-adic convergence we want to prove that $\operatorname{sgn}(w(z))$ divides $w_P(z)$ in $\T_z(K_P)$. Remark that the possible $P$-adic poles are the zeros of $\operatorname{sgn}(w(z))\in \F_q[z]$ hence elements of $\overline{\F}_q$. We will prove that the meromorphic series $\dfrac{w_P(z)}{\operatorname{sgn}(w(z))}$ does not have a pole in $\overline{\F}_q$.

\begin{theoreme}\label{th:paspole}
 The meromorphic series $\dfrac{w_P(z)}{\operatorname{sgn}(w(z))}$ does not have a pole in $\overline{\F}_q$. In other words, we have the convergence $\dfrac{w_P(z)}{\operatorname{sgn}(w(z))}\in \T_z(K_P)$.
\end{theoreme}

\begin{proof}
 Let $\zeta\in \overline{\F}_q$ be a root of $\operatorname{sgn}(w(z))$. Recall that $\mathscr{C}_\zeta=(1\otimes g_1,\ldots ,1\otimes g_m)\subseteq \F_q(\zeta)\otimes_{\F_q}\mathscr{O}_L$ if $\mathscr{C}=(g_1,\ldots, g_m)$. Then $\operatorname{Lie}_{F_\zeta}(\F_q(\zeta)\otimes_{\F_q} \oo)$ is an $A_\zeta$-lattice in $M_\infty$ and admits $\mathscr{C}_\zeta$ as an $A_\zeta$-basis.
Consider $(w_1,\ldots ,w_m)$ an $A_{\zeta}$-basis of $U_\zeta(F;\oo)=\operatorname{ev}_{\zeta} U(F^{(\zeta)};\widetilde{M})$ and $(w_1(z),\ldots ,w_m(z))\subseteq U(F^{(\zeta)};\widetilde{M})$ be such that $\operatorname{ev}_{\zeta} w_i(z)=w_i$ for $i=1,\ldots ,m$. Set $$W'(z)=\det_{\mathscr{C}_\zeta} (w_1(z),\ldots ,w_m(z))\in \widetilde{M}_{\infty}\backslash (z-\zeta)\widetilde{M}_{\infty}$$
and 
$$W_P'(z)=\det_{\mathscr{C}_\zeta}\left(\log_{\widetilde{F}^{(\zeta)},P} (\exp_{\widetilde{F}^{(\zeta)}}(w_1(z))),\ldots, \log_{\widetilde{F}^{(\zeta)},P} (\exp_{\widetilde{F}^{(\zeta)}}(w_m(z)))\right)\in \widetilde{M}_v.$$
Recall that the family $(1\otimes v_1(z),\ldots ,1\otimes v_m(z))$ is an $\widetilde{A}_{\zeta}$-basis of $U(\widetilde{F}^{(\zeta)};\widetilde{M})$. Let us set
$$W(z)=\det_{\mathscr{C}_\zeta}(1\otimes v_1(z),\ldots ,1\otimes v_m(z))=1\otimes w(z)\in \widetilde{M}_{\infty},$$ 
$$W_P(z)=1\otimes w_P(z)\in \widetilde{M}_v,$$
and $\Delta=\det_{(1\otimes v_1(z),\ldots ,1\otimes v_m(z))}(w_1(z),\ldots ,w_m(z))\in \widetilde{A}_{\zeta}.$\\
From the equality $$W'(z)=\Delta W(z)$$
we obtain 

$$1\otimes \underbrace{L(\widetilde{F}/\widetilde{\oo})}_{\in \T_z(K_\infty)}=1\otimes \dfrac{w(z)}{\operatorname{sgn}(w(z))}=\dfrac{1\otimes w(z)}{1\otimes \operatorname{sgn}(w(z))}=\dfrac{W(z)}{1\otimes \operatorname{sgn}(w(z))}=\dfrac{W'(z)}{\Delta (1\otimes \operatorname{sgn}(w(z))}.$$
Since $1\otimes L(\widetilde{F}/\widetilde{\oo})$ does not have a pole in $\overline{\F}_q$ and $W'(z)$ is not divisible by $z-\zeta$, we obtain that $\Delta (1\otimes \operatorname{sgn}(w(z))$ is not divisible by $z-\zeta$. From the equality
$$\dfrac{W_P(z)}{1\otimes \operatorname{sgn}(w(z))}=\dfrac{W_P'(z)}{\Delta(1\otimes \operatorname{sgn}(w(z))}$$

we can evaluate at $z=\zeta$ so $\zeta$ is not a pole of $\dfrac{w_P(z)}{\operatorname{sgn}(w(z))}$.

Finally, the $P$-adic $L$-series is a meromorphic series without any pole in $\overline{\F}_q$: it is an element of $\T_z(K_P)$.
\end{proof}

\begin{cor}[$P$-adic $L$-series for $P^{-1}\widetilde{E}P$]\label{classFz} Consider the $t$-module $F=P^{-1}EP$. Consider $(v_1(z),\ldots, v_r(z))\subseteq U(\widetilde{F},\mathscr{O}_L[z])$ an $\widetilde{A}$-basis of $U(\widetilde{F};\widetilde{\oo})$, $\mathscr{C}$ an $\widetilde{A}$-basis of $\operatorname{Lie}_{\widetilde{F}}(\widetilde{\oo})$. Then the following product converges in $\T_z(K_P)$  $$L_P(\widetilde{E}/\widetilde{\oo})=\Prod\limits_{Q\neq P}\dfrac{\left[\operatorname{Lie}_{\widetilde{E}}(\widetilde{\oo}/Q\widetilde{\oo})\right]_{\widetilde{A}}}{\left[E(\widetilde{\oo}/Q\widetilde{\oo})\right]_{\widetilde{A}}}$$ where the product runs over all the monic irreducible polynomials $Q$ of $A$ different from $P$. Further, we have the equality:
$$L_P(\widetilde{E}/\widetilde{\oo})=\dfrac {\det_{\mathscr{C}}\left(\log_{\widetilde{F},P} \exp_{\widetilde{F}}v_1(z),\ldots,\log_{\widetilde{F},P} \exp_{\widetilde{F}}v_m(z)\right)}{\operatorname{sgn} \left(\det_{\mathscr{C}}\left(v_1(z),\ldots,v_m(z)\right)\right)}.$$

\end{cor}
We can then define the $P$-adic $L$-series associated with $E$ and $\mathscr{O}_L$:
$$L_P(E/\mathscr{O}_L)=\operatorname{ev}_{z=1,P}L_P(\widetilde{E},\widetilde{\oo})=\Prod\limits_{Q\neq P}\dfrac{\left[\operatorname{Lie}_E(\mathscr{O}_L/Q\mathscr{O}_L)\right]_{A}}{\left[E(\mathscr{O}_L/Q\mathscr{O}_L)\right]_{A}} \in K_P.$$
We also have the following equalities from \cite[Proposition 3.3]{Admissible}:

$$L_P(\widetilde{E}/\widetilde{\oo})=\Prod\limits_{\mathfrak{P}\nmid P}\dfrac{\left[\operatorname{Lie}_{\widetilde{E}}(\widetilde{\oo}/\mathfrak{P}\widetilde{\oo})\right]_{\widetilde{A}}}{\left[E(\widetilde{\oo}/\mathfrak{P}\widetilde{\oo})\right]_{\widetilde{A}}}\in \T_z(K_P)$$ where the product runs over all the primes of $\oo$ not dividing $P$, and
$$L_P(E/\oo)=\Prod\limits_{\mathfrak{P}\nmid P}\dfrac{\left[\operatorname{Lie}_E(\mathscr{O}_L/\mathfrak{P}\mathscr{O}_L)\right]_{A}}{\left[E(\mathscr{O}_L/\mathfrak{P}\mathscr{O}_L)\right]_{A}} \in K_P.$$

\subsection{A \texorpdfstring{$P$}{}-adic class formula associated with the \texorpdfstring{$t$}{}-module \texorpdfstring{$P^{-1}EP$}{}}

The next step is to introduce a $P$-adic regulator and obtain a $P$-adic class formula. We begin with $P$-twisted $t$-modules. Recall that $\mathscr{C}$ is a fixed ${A}$-basis of $\operatorname{Lie}_{{E}}({\oo})=\operatorname{Lie}_F(\oo)$.

\begin{defi}
 Consider $V\subseteq U(F;\oo)$ a sub-$A$-lattice and let $(v_1,\ldots ,v_m)$ be an $A$-basis of $V$. Then we define the $P$-adic regulator associated with $V$ by
$$R_P(V)=\dfrac{\det_{\mathscr{C}}\left(\log_{F,P}(\exp_F (v_1)),\ldots ,\log_{F,P}(\exp_F (v_m))\right)}{\operatorname{sgn}(\det_{\mathscr{C}}(v_1,\ldots ,v_m))}\in K_P$$
which is independent of the choice of the basis of $V$ and of $\operatorname{Lie}_F(\oo)$. 
\end{defi}

\begin{theoreme}\label{classF}[$P$-adic class formula for $P^{-1}EP$]
We have the following equality in $K_P$:
    $$L_P(E/\oo)=R_P(U(F;\oo))\left[H(F;\oo)\right]_A=R_P(U_{\operatorname{st}}(F;\oo)).$$
\end{theoreme}

\begin{proof}
We use notation as in the proof of Theorem \ref{th:paspole} with $\zeta=1$. In particular $(v_1,\ldots, v_m)$ is an $A$-basis of $U(F;\oo)$. Consider $(u_1,\ldots, u_m)$ an $A$-basis of $U_{\operatorname{st}}(F;\oo)$. Denote by $b_i=\exp_F(u_i)\in F(\oo)$ for $i=1,\ldots ,m$ and by $a_i=\exp_F(v_i)\in F(\oo)$ for $i=1,\ldots, m$.
We have 
$$L(\widetilde{F}/\widetilde{\oo})=\dfrac{w'(z)}{f(z)\Delta}.$$
 As the $L$-series $L(F/\oo)$ (at infinity) is equal to 
 $$\dfrac{\det_{\mathscr{C}}(u_1,\ldots ,u_m)}{\operatorname{sgn}(\det_{\mathscr{C}}(u_1,\ldots ,u_m))}=\dfrac{\operatorname{ev}_{z=1} (w'(z))}{\operatorname{sgn}(\det_{\mathscr{C}}(u_1,\ldots ,u_m))}=\operatorname{ev}_{z=1} \dfrac{w'(z)}{f(z)\Delta}$$ 
we first have 
\begin{equation}\label{eq:sgn}\operatorname{ev}_{z=1} (f(z)\Delta)=\operatorname{sgn}(\det_{\mathscr{C}}(u_1,\ldots ,u_m))\in \F_q^{\ast}.\end{equation} Let us consider $P_1=\operatorname{Mat}_{(v_1,\ldots ,v_m)}(u_1,\ldots ,u_m)\in M_m(A)$. By \cite[Theorem 1]{units} we have $\dfrac{\det(P_1)}{\operatorname{sgn}(\det(P_1))}=\left[H(F;\oo)\right]_A$. Moreover, we have: $$P_1=\operatorname{Mat}_{(\log_{F,P}(a_1),\ldots ,\log_{F,P}(a_m))}(\log_{F,P}(b_1),\ldots ,\log_{F,P}(b_m)).$$
Then we have in $K_P$:
\begin{equation}\label{r1}\det (P_1){\det_{\mathscr{C}}(\log_{F,P}(a_1),\ldots ,\log_{F,P}(a_m)}={\det_{\mathscr{C}}(\log_{F,P}(b_1),\ldots ,\log_{F,P}(b_m)}).\end{equation}
From the following equality in $K_\infty$:
$$\det (P_1)\det_{\mathscr{C}}(v_1,\ldots ,v_m)=\det_{\mathscr{C}}(u_1,\ldots ,u_m)$$ we deduce by comparing signs:
\begin{equation}\label{r2}\operatorname{sgn} (\det (P_1))\operatorname{sgn} (\det_{\mathscr{C}}(v_1,\ldots ,v_m))=\operatorname{sgn}(\det_{\mathscr{C}}(u_1,\ldots ,u_m)).\end{equation}
We finally have:

$$\begin{aligned}L_P(E/\oo)&=\operatorname{ev}_{z=1,P}\left(\dfrac{w'_P(z)}{f(z)\Delta}\right)\\
&=\dfrac{\det_{\mathscr{C}}(\log_{F,P}(b_1),\ldots ,\log_{F,P}(b_m))}{\operatorname{sgn}( \det_{\mathscr{C}}(u_1,\ldots ,u_m))} \text{ by Equality \eqref{eq:sgn} },\\
&=\det (P_1)\dfrac{\det_{\mathscr{C}}(\log_{F,P}(a_1),\ldots ,\log_{F,P}(a_m))}{\operatorname{sgn}( \det_{\mathscr{C}}(u_1,\ldots ,u_m))} \text{ by Equality } \eqref{r1},\\
&= \dfrac{\det(P_1)}{\operatorname{sgn} (\det (P_1))}\dfrac{\det_{\mathscr{C}}(\log_{F,P}(a_1),\ldots ,\log_{F,P}(a_m))}{\operatorname{sgn}(\det_{\mathscr{C}}(v_1,\ldots ,v_m))} \text{ by Equality } \eqref{r2},\\
&=R_P(U(F;\oo))\left[H(F;\oo)\right]_A\end{aligned}$$
and the second line equals $R_P(U_{\operatorname{st}}(E;\oo))$.

\end{proof}

\subsection{A \texorpdfstring{$P$}{}-adic class formula}
So far we've worked mainly with the $t$-module $F$, and now we want to link everything to the $t$-module $E$.

Set $h(z)=\left[\widetilde{E}(\widetilde{\oo}/P\widetilde{\oo})\right]_{\widetilde{A}}\in A[z]$ and $h(1)=\left[E(\oo/P\oo)\right]_{A}\in A\backslash\{0\}$. Consider $s\in \N$ such that $\partial_{\widetilde{E}}(h(z)^{q^s})=h(z)^{q^s}I_d$ (e.g., $s$ such that $q^s\geq d$) and denote by $g(z)=h(z)^{q^s}\in A[z]$. By Proposition \ref{fitt}, we have for all $b(z)\in \mathscr{O}_L[z]^d$: 
$$\widetilde{E}_{g(z)}(b(z))\in P\mathscr{O}_L[z]^d$$ and for all $b\in \oo^d$:
$$E_{g(1)}(b)\in P\oo^d.$$
Then by Corollary \ref{cvlogp} and the above discussion, we can define the following maps:
$$\begin{aligned}\Log_{\widetilde{E},P}:\Omega_z&\rightarrow \dfrac{1}{g(z)} \T_z(L_P)^d  \\
x&\mapsto \dfrac{1}{g(z)}\log_{\widetilde{E},P} (\widetilde{E}_{g(z)}(x))\end{aligned}$$ and
$$\begin{aligned}\Log_{{E},P}:\Omega&\rightarrow L_P^d  \\
x&\mapsto \dfrac{1}{g(1)}\log_{{E},P} ({E}_{g(1)}(x))\end{aligned}.$$
Moreover, if $x\in \Omega_z^+$, we have $\log_{\widetilde{E},P} (\widetilde{E}_{g(z)}(x))=g(z)\log_{\widetilde{E},P}(x)$ in $\T_z(L_P)^d$. We obtain the following equality in $\T_z(L_P)^d$ for such $x$:
$$\log_{\widetilde{E},P} (x)=\Log_{\widetilde{E},P}(x)$$ thus the map $\Log_{\widetilde{E},P}$ extends the map $\log_{\widetilde{E},P}$ from $\Omega_z^+$ to $\Omega_z$. The same applies without $z$.

\begin{lem} We have the following properties.
\begin{enumerate}
\item 
    For all $a\in {A}[z]$ and $x\in \Omega_z$ we have the following equality in $\dfrac{1}{g(z)}\T_z(L_P)^d$:
    $$\partial_{\widetilde{E}}(a)\Log_{\widetilde{E},P} (x)=\Log_{\widetilde{E},P} (\widetilde{E}_a(x)).$$
   
    \item For all $x\in \Omega_z$ we have the equality in $L_P^d$:
    $$\operatorname{ev}_{z=1,P} (\Log_{\widetilde{E},P}(x))=\Log_{E,P} (\operatorname{ev}_{z=1,P} (x)).$$
     \item For all $a\in A$ and $x\in \Omega$ we have the following equality in $L_P^d$:
    $$\partial_E(a)\Log_{E,P} (x)=\Log_{E,P} (E_a(x)).$$
    \end{enumerate}

\end{lem}
\begin{proof}
    \begin{enumerate}
        \item We have the following equalities in $\T_z(L_P)^d$ for all $x\in \Omega_z$ and $a\in A[z]$:
        $$\begin{aligned}g(z)\Log_{\widetilde{E},P} (\widetilde{E}_a(x))=\log_{\widetilde{E},P} (\underbrace{\widetilde{E}_{g(z)}(\widetilde{E}_a(x))}_{v_P>0})
        &=\partial_{\widetilde{E}}(a)\log_{\widetilde{E},P} (\widetilde{E}_{g(z)}(x))\\
        &=\partial_{\widetilde{E}}(a)g(z)\Log_{\widetilde{E},P} (x).\end{aligned}$$

        \item We have the following equality in $\Omega^+$ for all $x\in\Omega_z$:
        $$\operatorname{ev}_{z=1,P} (\widetilde{E}_{g(z)}(x))=E_{g(1)}(\operatorname{ev}_{z=1,P} (x)).$$
        Then we have the following equalities:
        $$\begin{aligned}
            \operatorname{ev}_{z=1,P} (g(z)\operatorname{Log}_{\widetilde{E},P} (x))=\operatorname{ev}_{z=1,P}(\log_{\widetilde{E},P}\underbrace{(\widetilde{E}_{g(z)}(x)))}_{v_P>0}
            &=\log_{E,P} (E_{g(1)} (\operatorname{ev}_{z=1,P}(x)))\\
            &=g(1)\Log_{E,P} (\operatorname{ev}_{z=1,P}(x)).
        \end{aligned}$$

        \item The proof is similar as for the first assertion.
    \end{enumerate}
\end{proof}

\begin{prop}\label{prop:logzinjective}
    The logarithm map $\log_{\widetilde{E},P}$ is injective on $\Omega_z^+$.

\end{prop}

We begin with the following lemma.
\begin{lem}\label{lem:Enonnul} For all $a\in A\backslash\{0\}$ and for all $x\in \T_z(L_P)^d\backslash\{0\}$ we have $\widetilde{E}_a(x)\neq 0$.
\end{lem}

\begin{proof}
Fix $a\in A\backslash\{0\}$ and $x\in \T_z(K_P)^d$ such that $E_a(x)=0$. We can suppose without loss of generality that $\partial_E(a)=aI_d$, even if it means applying $\widetilde{E}_{a^{q^{s}-1}}$ to $\widetilde{E}_a(x)$. We can also assume that $z$ does not divide $x$ in $\T_z(L_P)^d$.
Denote by $\widetilde{E}_a=a+\Sum\limits_{i=1}^{r_a} \widetilde{E}_{a,i}\tau^i$ with $\widetilde{E}_{a,i}\in zM_d(\oo[z])$ for all $i=1,\ldots, r_a$, and $x=\Sum\limits_{n\geq 0} y_n z^n$ with $a_n\in L_P^d$ and $y_0\neq 0$.
We have $E_a(x)=ay_0 \text{ mod }z\T_z(L_P)^d\neq 0\text{ mod } z\T_z(L_P)^d$. Hence $E_a(x)\neq 0$.
\end{proof}

\begin{proof}[Proof of Proposition \ref{prop:logzinjective}]
Let $x$ be in $ \Omega_z^+$ such that $\log_{\widetilde{E},P} (x)=0$. Then for all $s\in \N$ we have 
$$\partial_{\widetilde{E}}(P^{q^s})\log_{\widetilde{E},P}(x)=0=\log_{\widetilde{E},P} (\widetilde{E}_{P^{q^s}}(x)).$$
Since $v_P(x)>0$, we consider $s\in \N$ large enough such that 
${\widetilde{E}}_{P^{q^s}}(x)$ belongs to $\mathcal{D}_z^+$. For such an integer $s$ we obtain:
$$\widetilde{E}_{P^{q^s}}(x)=0$$ which implies that $x= 0$ by Lemma \ref{lem:Enonnul}.

\end{proof}

\begin{prop} The kernel of $\Log_{E,P}:\oo^d\rightarrow L_P^d$ consists exactly of the torsion points of $E(\oo)$, denoted by $E(\oo)_{\operatorname{Tors}}$.\end{prop}

\begin{proof}

Consider first $x\in \oo^d$ such that $\Log_{E,P} (x)=0$. Then we have $\log_{E,P} (E_{g(1)}(x))=0.$ Thus for all $n\geq 0$ we have: 
    $$\log_{E,P} (E_{P^{q^n}g(1)}(x))=\partial_E(P^{q^n})\log_{E,P} (E_{g(1)}(x))=0.$$ Since $v_P (E_{g(1)}(x))>0$, we can consider $n$ large enough such that $E_{P^{q^n}g(1)}(x)\in \mathcal{D}^+$. Then we find by applying the exponentiel map $\exp_{E,P}$ to $\log_{E,P} (E_{P^{q^n}g(1)}(x))$ that $0=E_{P^{q^ng(1)}}(x)$ so $x$ is a torsion point. 
    
    Conversely, suppose that there is a non-zero polynomial $a\in A$ such that $E_a(x)=0$. We also have $E_{a^{q^s}}(x)=0$. Then we have
    $$a^{q^s}\log_{E,P} (E_{g(1)}(x)) =\log_{E,P} (E_{a^{q^s}g(1)} (x))=\log_{E,P} (E_{g(1)}(E_{a^{q^s}}(x)))=0.$$ Since $a$ is non-zero, we obtain $\log_{E,P} (E_{g(1)}(x))=0$.
\end{proof}

Set $\mathscr{O}_{{L},P}=\oo\otimes_A {A}_P$ and $\widetilde{\mathscr{O}_{{L},P}}=\widetilde{\oo}\otimes_A A_P$. 
By \cite[Lemma 3.21]{units}, we can extend $E$ by continuity to a homomorphism of $\F_q$-algebras
    $$E:A_P\rightarrow M_d(\mathscr{O}_{{L},P})\{\{\tau\}\}.$$ We can also extend $\widetilde{E}$ by continuity to a homomorphism of $\F_q(z)$-algebras
    $$\widetilde{E}:\widetilde{A}_P\rightarrow M_d(\widetilde{\mathscr{O}_{{L},P}})\{\{\tau\}\}.$$  
In particular, the $A$-module $E(P\mathscr{O}_{L,P})$ inherits a structure of $A_P$-module.

Set $U'=g(1)U(E;\oo)$ and $U'_z=g(z)U(\widetilde{E};\widetilde{\oo})$ (the multiplication is of course in $\operatorname{Lie}_E(L_\infty)$ (resp. $\operatorname{Lie}_{\widetilde{E}}(\widetilde{L_\infty})$).
\begin{lem}\label{lem:sub}
We have the following properties.
\begin{enumerate}
\item We have that $U'$ and $U_z'$ are sub-lattices of $U(E;\oo)$ and $U(\widetilde{E};\widetilde{\oo})$ respectively.
\item We have that $U'$ and $U_z'$ are sub-lattices of $U(F;\oo)$ and $U(\widetilde{F};\widetilde{\oo})$ respectively.
\end{enumerate}
\end{lem}

\begin{proof}
\begin{enumerate}
    \item The first point is clear.
    \item We have to prove the inclusions $U'\subseteq U(F;\oo)$ and $U_z'\subseteq U(\widetilde{F};\widetilde{\oo}).$ We do it for $F$, the same arguments apply to $\widetilde{F}$. Let $x\in \operatorname{Lie}_E(L_\infty)$ be such that $\exp_E (x)\in \oo^d$, then:
    $$\exp_F (g(1)x)=P^{-1}\exp_{E} (P\partial_E(g(1))x)=P^{-1}E_{g(1)}\left({\exp_E(Px)}\right)\in P^{-1}. P\oo^d = \oo^d.$$

    \end{enumerate}
\end{proof}
By Lemma \ref{lem:sub}, $U'$ (resp. $U_{z}'$) is a common sub-$A$-lattice (resp. sub-$\widetilde{A}$-lattice) for $U(E;\oo)$ and $U(F;\oo)$ (resp. $U(\widetilde{E};\widetilde{\oo})$ and $U(\widetilde{F};\widetilde{\oo})$). We then have:
$$\left[U(F;\oo):U(E;\oo)\right]_A=\left[U(F;\oo):U'\right]_A\left[U':U(E;\oo)\right]_A=\dfrac{\left[U(F;\oo):U'\right]_A}{\left[U(E;\oo):U'\right]_A}\in K^\ast$$ and 
$$\left[U(\widetilde{F};\widetilde{\oo}):U(\widetilde{E};\widetilde{\oo})\right]_{\widetilde{A}}=\left[U(\widetilde{F};\widetilde{\oo}):U_z'\right]_{\widetilde{A}}\left[U_z':U(\widetilde{E};\widetilde{\oo})\right]_{\widetilde{A}}=\dfrac{\left[U(\widetilde{F};\widetilde{\oo}):U_z'\right]_{\widetilde{A}}}{\left[U(\widetilde{E};\widetilde{\oo}):U_z'\right]_{\widetilde{A}}}\in \widetilde{K}^\ast.$$

Let us define $P$-adic regulators associated to $U'$ as follows.
Let $(w_1,\ldots ,w_m)$ be an $A$-basis of $U'$. We set:

$$R_{P,E}(U')=\dfrac{\det_{\mathscr{C}}\left(\Log_{E,P} (\exp_E (w_1)),\ldots ,\Log_{E,P} (\exp_E (w_m))\right)}{\operatorname{sgn}(\det_{\mathscr{C}}( w_1,\ldots , w_m))}\in K_P$$ and
$$R_{P,F}(U')=\dfrac{\det_{\mathscr{C}}\left(\log_{F,P} (\exp_F (w_1)),\ldots ,\log_{F,P} (\exp_F (w_m))\right)}{\operatorname{sgn}(\det_{\mathscr{C}}( w_1,\ldots , w_m))}\in K_P.$$ 
These regulators do not depend of the choice of the basis $(w_1,\ldots,w_m)$. We can also define a $P$-adic regulator $R_P(U_{\operatorname{st}}(E;\oo))$ for $U_{\operatorname{St}}(E;\oo)$ and we have the following equality from the proof of Theorem \ref{classF}:

\begin{equation}\label{regst}R_P(U(E;\oo))\left[H(E;\oo)\right]_A=R_P(U_{\operatorname{St}}(E;\oo)).
\end{equation}
Similarly, we define the $P$-adic regulators $R_{P,\widetilde{E}}(U_z')$ and $R_{P,\widetilde{F}}(U_z')$ associated with $U_z'$ that are elements in $\T_z(K_P)$ from Theorem \ref{th:paspole}.
\begin{lem}
    We have the following equalities:
\begin{equation}\label{egreg}
    R_{P,\widetilde{E}}(U_z')=R_{P,\widetilde{F}}(U_z') \text{ and }  R_{P,{E}}(U')=R_{P,{F}}(U').
\end{equation}
\end{lem}

\begin{proof}We prove the first equality. We want to prove that for all $u(z)\in U(\widetilde{E};\mathscr{O}_L[z])$ we have the following equality in $\T_z(L_P)^d$:
 $$\log_{\widetilde{F},P} (\exp_{\widetilde{F}}(u(z)))=\Log_{\widetilde{E},P}
 (\exp_{\widetilde{E}}(u(z))).$$
 Formally (i.e., in $L[[z]]^d)$ it is true, and the two quantities belong to $\T_z(L_P)^d$. Then the first equality of the lemma is clear.
 
 To prove the second equality, consider $(u_1,\ldots ,u_m)$ an $A$-basis of $U_{\operatorname{St}}(E;\oo)$.From the first equality of the lemma, we have the following equality in $\T_z(K_P)$:
 $$\det_{\mathscr{C}}(\log_{\widetilde{F},P} (\exp_{\widetilde{F}} (g(z)u_i(z))),i=1,\ldots, m)=\det_{\mathscr{C}}(\Log_{\widetilde{E},P} (\exp_{\widetilde{E}} (g(z)u_i(z))),i=1,\ldots,m).$$
By evaluating at $z=1$ we obtain
 $$\det_{\mathscr{C}}(\log_{{F},P} (\exp_{{F}} (g(1)u_i)), i=1,\ldots,m)=\det_{\mathscr{C}}(\Log_{{E},P} (\exp_{{E}} (g(1)u_i)),i=1,\ldots,m).$$
 Consider now $(v_1,\ldots, v_m)$ an $A$-basis of $U(E,\oo)$, then $(g(1)v_1,\ldots, g(1)v_m)$ is an $A$-basis of $U'$. Write
 $Q=\operatorname{Mat}_{(v_1,\ldots, v_m)}(u_1,\ldots, u_m)\in \operatorname{Gl}_m(A)$.
 We then have
 $$\begin{aligned}
     \det_{\mathscr{C}}(\log_{F,P}&(\exp_F (g(1)v_i)),i=1,\ldots, m)\\&= 
 \det (Q)^{-1}  \det_{\mathscr{C}}(\log_{F,P}(\exp_F (g(1)u_i)),i=1,\ldots, m)\\
 &=\det (Q)^{-1}\det_{\mathscr{C}}(\Log_{{E},P} (\exp_{{E}} (g(1)u_i)),i=1,\ldots,m)\\
 &=\det_{\mathscr{C}}(\Log_{E,P}(\exp_E (g(1)v_i)),i=1,\ldots, m).
  \end{aligned}$$
We proved that the equality is true for one basis of $U'$ and since the $P$-adic regulators does not depend of the choice of the basis, the equality is true.
\end{proof}

\begin{lem}\label{reg}
    \begin{enumerate}
        \item We have the following equalities: 
        $$R_{P,\widetilde{F}}(U_z')=R_P(U(\widetilde{F};\widetilde{\oo}))\left[U(\widetilde{F};\widetilde{\oo}):U_z'\right]_{\widetilde{A}}$$ and $$R_{P,\widetilde{E}}(U_z')=R_P(U(\widetilde{E};\widetilde{\oo}))\left[U(\widetilde{E};\widetilde{\oo}):U_z'\right]_{\widetilde{A}}$$
        \item We have the following equalities:
        $$R_{P,{F}}(U')=R_P(U({F};{\oo}))\left[U({F};{\oo}):U'\right]_{{A}}$$ and $$R_{P,{E}}(U')=R_P(U({E};\widetilde{\oo}))\left[U({E};{\oo}):U'\right]_{{A}}$$
    \end{enumerate}
\end{lem}
\begin{proof}
We only need to prove one of the equalities; all the others can be proven in a similar way. Let us prove the third one. By the structure theorem for finitely generated modules over a principal ideal domain, let us pick an $A$-basis $(v_1,\ldots ,v_m)$ of $U(F;\oo)$ and $a_1,\ldots ,a_m\in A$ such that $(a_1v_1,\ldots ,a_mv_m)$ is an $A$-basis of $U'$. Then
$$\begin{aligned}R_{P,F} (U') &=\dfrac{\det_{\mathscr{C}}(\log_{F,P} (\exp_F (a_1w_1)),\ldots ,\log_{F,P} (\exp_F (a_mw_m)))}{\operatorname{sgn}\left(\det_{\mathscr{C}}(\log_{F,P} (\exp_F (a_1w_1)),\ldots ,\log_{F,P} (\exp_F (a_mw_m)))\right)} \\ &=\dfrac{a_1\ldots a_m}{\operatorname{sgn} (a_1\ldots a_m)}R_P(U(F;\oo))\\
    &=\left[U({F};{\oo}):U'\right]_{{A}}R_P(U(F;\oo)).\end{aligned}$$
\end{proof}

The link between objects associated to $F$ and those associated to $E$ is contained in the local factor at $P$, and given by the following equalities from \cite[Lemma 3.4]{Admissible}:
\begin{equation}\label{facteur}
z_P(E/\oo)=\left[U(F;\oo):U(E;\oo)\right]_A\dfrac{\left[H(E;\oo)\right]_A}{\left[H(F;\oo)\right]_A}.
\end{equation}
 and
 \begin{equation}\label{facteurz}
z_P(\widetilde{E}/\widetilde{\oo})=\left[U(\widetilde{F};\widetilde{\oo}):U(\widetilde{E};\widetilde{\oo})\right]_{\widetilde{A}}.
\end{equation}

We can state one of the main results of this work.
\begin{theoreme}[$P$-adic class formula]\label{maincf}
 We have the $P$-adic class formula for $\widetilde{E}$:
    $$z_P(\widetilde{E}/\widetilde{\oo})L_P(\widetilde{E}/\widetilde{\oo})=R_P(U(\widetilde{E};\widetilde{\oo}))$$
    and the class formula for $E$:
    $$z_P(E/\oo)L_P(E/\oo)=R_{P}(U(E;\oo))\left[H(E;\oo))\right]_{A}=R_P(U_{\operatorname{st}}(E;\oo)).$$
\end{theoreme}

\begin{proof} 
Let us start with the following equality from Corollary \ref{classFz}:
$$L_P(\widetilde{E}/\widetilde{\oo})=R_P(U(\widetilde{F};\widetilde{\oo})).$$
Then:
$$\begin{aligned}z_P(\widetilde{E}/\widetilde{\oo})L_P(\widetilde{E}/\widetilde{\oo})&=\dfrac{\left[U(\widetilde{F};\widetilde{\oo}):U_z'\right]_{\widetilde{A}}}{\left[U(\widetilde{E};\widetilde{\oo}):U_z'\right]_{\widetilde{A}}}R_P(U(\widetilde{F};\widetilde{\oo})) \text{ by Equality } \eqref{facteurz},\\
&=\dfrac{R_{P,\widetilde{F}}(U_z')}{\left[U(\widetilde{E};\widetilde{\oo}):U_z'\right]_{\widetilde{A}}} \text{ by Lemma } \ref{reg},\\
&=\dfrac{R_{P,\widetilde{E}}(U_z')}{\left[U(\widetilde{E};\widetilde{\oo}):U_z'\right]_{\widetilde{A}}} \text{ by Equality } \eqref{egreg},\\
&=R_P(U(\widetilde{E};\widetilde{\oo}))  \text{ by Lemma } \ref{reg}.\end{aligned}$$

Recall Theorem \ref{classF}:
$$L_P(E/\oo)=R_P(U(F;\oo))\left[H(F;\oo)\right]_A.$$

Then:

$$\begin{aligned}z_P(E/\oo)&L_P(E/\oo)\\
&=\dfrac{\left[U(F;\oo):U')\right]_A}{\left[U(E;\oo):U'\right]_A}\dfrac{\left[H(E;\oo)\right]_A}{\left[H(F;\oo)\right]_A}R_P(U(F;\oo))\left[H(F;\oo)\right]_A \text{ by Equality } \eqref{facteur}, \\
&=\dfrac{R_{P,F}(U')}{\left[U(E;\oo):U'\right]_A}\left[H(E;\oo)\right]_A \text{ by Lemma } \ref{reg},\\
&=\dfrac{R_{P,E}(U')}{\left[U(E;\oo):U'\right]_A}\left[H(E;\oo)\right]_A \text{ by Equality } \eqref{egreg},\\
&=R_P(U(E;\oo))\left[H(E;\oo)\right]_A \text{ by Lemma } \ref{reg}.
\end{aligned}$$

\end{proof}

\subsection{Vanishing of the \texorpdfstring{$P$}{}-adic \texorpdfstring{$L$}{}-series}
We keep the notation as in Theorem \ref{th:paspole}. In particular, $(v_1,\ldots, v_m)$ is an $A$-basis of $U_{\operatorname{st}}(E;\oo)$ and $(u_1,\ldots, u_m)$ is an $A$-basis of $U(E;\oo)$.

\begin{prop}\label{prop:noyaustark}
 Suppose that there exists a non-zero element $x\in U_{\operatorname{St}}(E;\mathscr{O}_L)$ such that $\exp_E(x)=0$. Then
    $$L_P(E/\oo)=0.$$
\end{prop}

\begin{proof}
Write $x=\Sum\limits_{i=1}^{m} a_i v_i$ with $a_i\in A$ and suppose without loss of generality that $a_1\neq 0$. Then $x=x(1)$ with $x(z)=\Sum\limits_{i=1}^ma_i v_i(z)\in U(\widetilde{E},\mathscr{O}_L[z])$. We have:
$$\begin{aligned}&\det_{\mathscr{C}}(\Log_{\widetilde{E},P}(\exp_{\widetilde{E}}(v_i(z))),i=1,\ldots, m)\\
=&\dfrac{1}{a_1}\det_{\mathscr{C}}(\Log_{\widetilde{E},P}(\exp_{\widetilde{E}}(x(z))),\Log_{\widetilde{E},P}(\exp_{\widetilde{E}}(v_i(z))),i=2,\ldots,m).\end{aligned}$$
Since $\exp_E(x)=0$, we have:
$$\operatorname{ev}_{z=1,P} (\Log_{\widetilde{E},P} (\exp_{\widetilde{E}} x(z)))=\Log_{E,P} (\exp_E(x))=0.$$
We then conclude that $R_P(U_{\operatorname{St}}(E;\oo))=0$, so 
$L_P(E/\oo)=0$ from the $P$-adic class formula \ref{maincf}.

\end{proof}

\begin{theoreme}\label{thm:pasinjectif}
    If the exponential map $\exp_E:L_\infty^d\rightarrow L_\infty^d$ is not injective, then we have $$L_P(E/\mathscr{O}_L)=0.$$
\end{theoreme}

\begin{proof}
Let $x\in L_\infty^d$ be non-zero such that $\exp_E (x)=0$. There exists $a\in A\backslash\{0\}$ such that $ax\in U_{\operatorname{st}}(E;\oo)$ and we  have $\exp_E(ax)=$. By Proposition \ref{prop:noyaustark} we have $L_P(E/\oo)=0$.
\end{proof}
We believe that the converse statement holds.

\begin{conj}\label{conj:1}The $P$-adic $L$-series is non-zero if and only if the exponential map $\exp_E:L_\infty^d\rightarrow L_\infty^d$ is injective.
\end{conj}

By \cite[Corollary 3.24]{units}, it is true when $d=1$ (i.e., in the Drinfeld module case) and $L=K$.\\
Remark that in the case $\exp_E:L_\infty^d\rightarrow L_\infty^d$ is injective, which we will call the totally real case, then $\mathscr{U}(E;\oo)=\exp_E (U(E; \oo))\subseteq E(\oo)$ is a free $A$-module of rank $m$, and the family $(\Log_{E,P} (\exp_E (u_i)),i=1,\ldots, m)$ is $A$-free. We would like to have that this family is $A_P$-free to obtain the non-vanishing of the $P$-adic $L$-series. Set:\\
$U(E;P\oo)=\{x\in \operatorname{Lie}_E(L_\infty)\ens \exp_E(x)\in E(P\oo)\}$ and $\mathscr{U}(E;P\oo)=\exp_E (U(E;P\oo))$. Then we can state an equivalent of the Leopoldt's conjecture in \cite{leo}, introduced recently by Anglès in \cite[Section 6.3]{An} for the Carlitz module.

\begin{conj}[Conjecture A]\label{leopoldt} The $A_P$-rank of $\mathscr{U}(E;P\oo)$ is equal to the $A$-rank of $\mathscr{U}(E;\oo)$.
\end{conj}
This conjecture is clear in the case $d=1$ and $L=K$. For further discussion of this conjecture, the reader may wish to see the paper by Anglès, Bosser and Taelman \cite{leocarlitz} where this conjecture is proved in the case of the Carlitz module defined on the $P$th cyclotomic extension.

In the totally real case, the non-vanishing of the $P$-adic $L$-series $L_P(\widetilde{E}/\widetilde{\oo})$ at $z=1$ is equivalent to the previous Leopoldt conjecture. This result can be seen as an analog to the following result from \cite{col}.

\begin{theoreme} Let $F$ be a totally real extension of $\Q$. Then the $p$-adic zeta function $\zeta_{F,p}(s)$ has a simple pole at $s=1$ if and only if the (usual) Leopoldt conjecture is true for $(F,p)$.
\end{theoreme}

\begin{defi}We call order of vanishing of the $P$-adic $L$-series and denote by \\ $\operatorname{ord}_{z=1}L_P(\widetilde{E}/\widetilde{\oo})$, the greatest integer $n$ such that $(z-1)^n$ divides  $L_P(\widetilde{E}/\widetilde{\oo})$.
\end{defi}

Fox example if the exponential map $\exp_E:L_\infty^d\rightarrow L_\infty^d$ is not injective, then for all $P$ we have $\operatorname{ord}_{z=1}L_P(\widetilde{E}/\widetilde{\oo})\geq 1$ and the previous conjecture tells us that
$\operatorname{ord}_{z=1}L_P(\widetilde{E}/\widetilde{\oo})=0$ if and only if $\exp_E$ in injective.\\
Here is a list of conjectures.

\begin{conj}[Conjecture B]: The vanishing order of the $P$-adic $L$-series at $z=1$ is independent of $P$.\end{conj}
Caruso and Gazda \cite{cg} have already conjectured this in the context of Anderson motives. Caruso, Gazda and the author proved this conjecture in the case $L=K$ and $d=1$, see \cite[Theorem 2.17]{cgl}. 
\begin{conj}[Conjecture C]
 We have $\operatorname{ord}_{z=1}L_P(\widetilde{E}/\widetilde{\oo})\leq [L:K] r_{\Omega_E}d$ where $r_{\Omega_E}$ is the rank of the period lattices $\Omega_E$ associated with $E$.
\end{conj}
We prove conjecture $C$ in section \ref{section:facile} in the case $d=1$ and $L=K$.

\section{The multi-variable setting}\label{section:variables}
We keep the notation from Sections \ref{section:notation}, \ref{section:infini}, \ref{section:padique} and from the Introduction. In particular $L/K$ is a finite field extension of degree $n$ and $\oo$ denotes the integral closure of $A$ in $L$.

The aim of this section is to extend the previous constructions to the case where the constant field is no longer $\F_q$ but $\F_q(t_1,\ldots,t_s)$ where the $t_i$ are new variables. One of the interests of these constructions is that in many cases, we can reduce the study of certain $t$-modules $E:\F_q[\theta]\rightarrow M_d(\oo)\{\tau\}$ to the study of Drinfeld modules $\phi:\F_q(t_1,\ldots ,t_s)[\theta] \rightarrow \mathscr{O}_L(t_1,\ldots, t_s)\{\tau\}$ simpler to understand. For an application to the study of the tensor power of the Carlitz module $C^{\otimes n}$ reduced to the study of the Carlitz module $C$, see the work of Anglès, Pellarin and Tavares Ribeiro in \cite{federico}.

\subsection{Setup}
The goal of this section is to extend the developed theory to the multi-variable setting by replacing $\F_q$ by $k=\F_q(t_1,\ldots, t_s)$. Recall that the Frobenius map acts as the identity on $k$. We keep the notation in the Introduction and we introduce the following notation.\\
$\bullet$ $A_s[z]\simeq k[z]\otimes_k A_s$,\\
$\bullet$ $\widetilde{A}_s=k(z)\otimes_{k}A_s$,\\
$\bullet$ $w$: a place of $K$ ($w=v_P$ a finite place or $w=v_\infty$ the infinite place),\\
$\bullet$ $\pi_w$: a uniformiser of $w$ ($\pi=P$ if $w=v_P$ and $\pi=\frac{1}{\theta}$ if $w=v_\infty$),\\
$\bullet$ $K_w=\F_w((\pi_w))$ denoted by $K_w=K_\infty$ if $w=v_\infty$ and $K_w=K_P$ if $w=v_P$,\\
$\bullet $ $\F_w$: the residue field associated with $w$ i.e., $\F_w=\F_P$ if $w=v_P$ and $\F_w=\F_q$ if $w=v_\infty$,\\
$\bullet$ $L_w=L\otimes_K K_w$ i.e., $L_w=L_P$ if $w=v_P$ and $L_w=L_\infty$ else,\\
$\bullet$ $k_w=\F_w(t_1,\ldots, t_s)$,\\
$\bullet$ $K_{s,w}=kw((\pi_w))$ denoted by $K_{s,P}$ if $w=v_P$ and $K_{s,\infty}$ if $w=v_\infty$,\\
$\bullet$ $\widetilde{K_{s,w}}=k_w(z)((\pi_w)),$ \\ 
$\bullet$ $L_s=kL$,\\ 
$\bullet$ $L_{s,w}=L\otimes_K K_{s,w}$ denoted by $L_{s,P}$ if $w=v_P$,\\ 
$\bullet$ $\widetilde{L_{s,w}}=L\otimes_{K} \widetilde{K_{s,w}}$ denoted by $\widetilde{L_{s,P}}$ if $w=v_P$,\\
$\bullet$ $\mathscr{O}_{L,s}[z]\simeq k[z]\otimes_k \mathscr{O}_{L,s}$, \\ 
$\bullet$ $\widetilde{\mathscr{O}_{L,s}}=k(z)\otimes_k\mathscr{O}_{L,s}$.

We recall that every $x\in \widetilde{K_{s,\infty}}^\ast$ (resp. $\in K_{s,\infty}$) can be written uniquely as $x=\sum\limits_{n\geq N} x_n\frac{1}{\theta^n}$ with $N\in \Z$, $x_n\in k(z)$ (resp. $x_n\in K$) and $x_N\neq 0$. We call $x_N\in k(z)$ (resp. $k$) the sign of $x$ denoted by $\operatorname{sgn}(x)$.
We define the Tate algebra in variables $\underline{t}=(t_1,\ldots, t_s)$: $$\T_{s}(K_w)=\left\{\Sum\limits_{{n}\in \N^s} a_{{n}}\underline{t}^{{n}}\in K_w[[\underline{t}]] \ens a_n\in K_w, \lim\limits_{n\rightarrow +\infty}w(a_{n})=+\infty\right\}$$
where $\underline{t}^n=t_1^{n_1}\ldots t_s^{n_s}$ if $n=(n_1,\ldots, n_s)\in \N^s$. This is the completion of $K[t_1,\ldots, t_s]$ with respect to the Gauss norm associated with $w$. We set:
 $$\T_{s}(L_w)=L\otimes_K \T_s(K_w).$$
 
 An Anderson $t$-module $E$ of dimension $d$ over $\mathscr{O}_{L,s}$ is a non-constant $k$-algebra homomorphism $E:A_s\rightarrow M_d({\oo}_{,s})$, $a\mapsto E_a=\Sum\limits_{i=0}^{r_a} E_{a,i}\tau^i\in M_d(\mathscr{O}_{L,s})\{\tau\}$ such that $(E_{\theta,0}^d-\theta I_d)^d=0$. We can consider $\widetilde{E}$, the $z$-twist of $E$, as in Section \ref{section:infini}. Following notation from Section \ref{section:notation}, we denote by $[M]_{A_s}$ the monic generator of $\operatorname{Fitt}_{A_s}(M)$ where $M$ is a torsion $A_s$-module of finite type, e.g., $M=E(\mathscr{O}_{L,s}/P\mathscr{O}_{L,s})$ and $M=\operatorname{Lie}_E(\mathscr{O}_{L,s}/P\mathscr{O}_{L,s})$.\\
 As in Proposition \ref{prop:exp} there exists a unique element $\exp_E\in M_d(L_s)\{\{\tau\}\}$ called the exponential map associated with $E$ and converging over $L_{s,\infty}^d$. Similarly, there exists a logarithm map $\log_E\in M_d(L_s)\{\{\tau\}\}$ as in Proposition \ref{prop:log}. 
\subsection{The \texorpdfstring{$\infty$}{}-case}
We can now define the module of units and class module in the multi-variable setting:

$$U(E;\mathscr{O}_{L,s})=\{x\in \operatorname{Lie}_E(L_{s,\infty})\ens \exp_E (x)\in E(\mathscr{O}_{L,s})\}$$ and the class module
$$H(E;\mathscr{O}_{L,s})=\dfrac{E(L_{s,\infty})}{E(\mathscr{O}_{L,s})+\exp_{{E}}(\operatorname{Lie}_{{E}}(L_{s,\infty}))}$$ both provided with $A_s$-module structure. We define the module of $z$-units:
$$U(\widetilde{E};\widetilde{\mathscr{O}_{L,s}})=\left\{x\in \operatorname{Lie}_{\widetilde{E}}(\widetilde{L_{s,\infty}}) \ens \exp_{\widetilde{E}}(x)\in \widetilde{E}(\widetilde{\mathscr{O}_{L,s}})\right\}$$
and the class module for the $z$-deformation:
$$H(\widetilde{E};\widetilde{\mathscr{O}_{L,s}})=\dfrac{\widetilde{E}(\widetilde{L_{s,\infty}})}{\widetilde{E}(\widetilde{\mathscr{O}_{L,s}})+\exp_{\widetilde{E}}(\operatorname{Lie}_{\widetilde{E}}(\widetilde{L_{s,\infty}}))}$$
both provided with $\widetilde{A_s}$-module structure. We define the module of $z$-units at the integral level:
$$U(\widetilde{E};\mathscr{O}_{L,s}[z])=\left\{x\in \operatorname{Lie}_{\widetilde{E}}(\T_z(L_{s,\infty})) \ens \exp_{\widetilde{E}}(x)\in\widetilde{E}(\mathscr{O}_{L,s}[z])\right\}$$
and finally the class module at the integral level
$$H(\widetilde{E};\mathscr{O}_{L,s}[z])=\dfrac{\widetilde{E}(\T_z(L_{s,\infty}))}{\widetilde{E}(\mathscr{O}_L[z])+\exp_{\widetilde{E}}(\operatorname{Lie}_{\widetilde{E}}(\T_z(L_{s,\infty})))}$$ both provided with $A_s[z]$-module structure. We have the following result from \cite[Proposition 2.8]{Flo}.
\begin{prop}\label{prop:classandunitsvariables}
The unit module $U(E;\mathscr{O}_{L,s})$ is an $A_s$-lattice in $\operatorname{Lie}_E(L_{s,\infty})$ and the module of $z$-units $U(\widetilde{E},\widetilde{\mathscr{O}_{L,s}})$ is an $\widetilde{A_s}$-lattice in $\operatorname{Lie}_{\widetilde{E}}(\widetilde{L_{s,\infty}})$.
\end{prop}

Denote by 
$$z_P(\widetilde{E}/\widetilde{\mathscr{O}_{L,s}})=
\dfrac{\left[\operatorname{Lie}_{\widetilde{E}}( \widetilde{\mathscr{O}_{L,s}}/{P}\widetilde{\mathscr{O}_{L,s}})\right]_{\widetilde{A}_s}}{\left[\widetilde{E}(\widetilde{\mathscr{O}_{L,s}}/{P}\widetilde{\mathscr{O}_{L,s}})\right]_{\widetilde{A}_s}}$$ the local factor associated with $\widetilde{E}$ at $P$ and 
$$z_P({E}/{\mathscr{O}_{L,s}})=
\dfrac{\left[\operatorname{Lie}_{{E}}( {\mathscr{O}_{L,s}}/{P}{\mathscr{O}_{L,s}})\right]_{{A}_s}}{\left[{E}({\mathscr{O}_{L,s}}/{P}{\mathscr{O}_{L,s}})\right]_{{A}_s}}$$ the local factor associated with $E$ at $P$. We have the following class formula for $t$-modules defined over $\mathscr{O}_{L,s}$, see \cite[Theorem 2.9]{Flo}.
\begin{theoreme}\label{classformulavariable} The following product 
$$L(\widetilde{E}/\widetilde{\mathscr{O}_{L,s}})=\Prod\limits_{ P} z_P(\widetilde{E}/\widetilde{\mathscr{O}_{L,s}})$$ where $P$ runs through the monic primes of $A$, converges in $\widetilde{K_{s,\infty}}$ and we have the class formula:

$$L(\widetilde{E}/\widetilde{\mathscr{O}_{L,s}})=\left[\operatorname{Lie}_{\widetilde{E}}\left(\widetilde{\mathscr{O}_{L,s}}\right):U(\widetilde{E};\widetilde{\mathscr{O}_{L,s}})\right]_{\widetilde{A_s}}.$$
\end{theoreme}

\subsection{The \texorpdfstring{$P$}{}-adic case}
We define the $P$-adic $L$-series in the multi-variable setting.
\subsubsection{The \texorpdfstring{$P$}{}-adic class formula}
All results from Section \ref{section:padique} remain valid by replacing $\F_q$ by $k$. In particular, we have the following $P$-adic class formula. 

\begin{theoreme}[$P$-adic class formula]\label{main0}We have the following assertions.
\begin{enumerate}
    \item 
 The infinite product 

$$L_P(\widetilde{E}/\widetilde{\mathscr{O}_{L,s}})=\Prod\limits_{Q\neq P} z_Q(\widetilde{E}/\widetilde{\mathscr{O}_{L,s}})$$ where $Q$ runs through the monic primes of $A$ different from $P$, converges in $\T_z(K_{s,P})$ and we have the class formula:

$$z_P(\widetilde{E}/\widetilde{\mathscr{O}_{L,s}})L_P(\widetilde{E}/\widetilde{\mathscr{O}_{L,s}})=R_P(U(\widetilde{E};\widetilde{\mathscr{O}_{L,s}})).$$
\item The infinite product 

$$L_P({E}/{\mathscr{O}}_{L,s})=\Prod\limits_{Q\neq P}z_Q({E}/{\mathscr{O}_{L,s}})$$ where ${Q}$ runs through the monic primes of $A$ different from $P$, converges in ${K}_{s,P}$ and we have the class formula:

$$z_P(E/\mathscr{O}_{L,s})L_P({E}/{\mathscr{O}}_{L,s})=R_P(U({E};{\mathscr{O}}_{L,s}))\left[H(E;\mathscr{O}_{L,s})\right]_{A_s}.$$
\end{enumerate}
\end{theoreme}

\begin{proof} The proof follows the same lines as the proof of \ref{th:paspole} by replacing $\F_q$ by $k$. We omit the details.\end{proof}

Denote by $U(E;P\mathscr{O}_{L,s})=\{x\in \operatorname{Lie}_E(L_{s,\infty})\ens \exp_E(x)\in E(P\mathscr{O}_{L,s})\}$ and consider the $A_s$-module $\mathscr{U}(E;\mathscr{O}_{L,s})=\exp_E(U(E;\mathscr{O}_{L,s}))$. Consider also the $A_{P,s}$-module $\mathscr{U}(E;P\mathscr{O}_{L,s})=\exp_E(U(E;P\mathscr{O}_{L,s}))$. Then the proof of Theorem \ref{thm:pasinjectif} is still valid in the multi-variable setting by replacing $\F_q$ by $k$.
\begin{prop}\label{main3} We have the following assertions.
    \begin{enumerate}
        \item If the exponential map $\exp_E:L_{s,\infty}^d\rightarrow L_{s,\infty}^d$ is not injective, then $L_P(E/\mathscr{O}_{L,s})=0$.

        \item Assume that the $A_s$-rank of $\mathscr{U}(E;\mathscr{O}_{L,s})$ and the $A_{P,s}$-rank of $\mathscr{U}(E;P\mathscr{O}_{L,s})$ are equal. Then $L_P(E/\oo)\neq0$ if and only if the exponential map $\exp_E:L_{s,\infty}^d\rightarrow L_{s,\infty}^d$ is injective.
    \end{enumerate}
\end{prop}

\subsubsection{The integral level}
In the work of \cite{federico}, given a $t$-module $E:A_s\rightarrow M_d(A_s)\{\tau\}$, they want to evaluate the variables $(t_1,\ldots, t_s)$ at some $\zeta\in\overline{\F}_q^s$. In this case, they need that all of the coefficients $E_{\theta,i}$ of $E_\theta$, for $i=0,\ldots, r$, can be evaluated at $\zeta$. This is possible if all the $E_{i,\theta}$ belong to $M_d(\F_q[t_1,\ldots, t_s]\oo)$. This is what we call the integral level.

We suppose now that: $E_\theta\in M_d(\mathscr{O}_{L}[t_1,\ldots,t_s])\{\tau\}$ i.e., we want to work at the integral level. 
\begin{theoreme}\label{thm:classeS} The $L$-series $L(\widetilde{E}/\widetilde{\mathscr{O}_{L,s}})$ converges in $\T_{s,z}(K_\infty)$ and we have the class formula:
$$L(\widetilde{E}/\widetilde{\mathscr{O}_{L,s}})=\dfrac{\det_{\mathscr{C}}\left(u_1(z),\ldots, u_m(z)\right)}{\sgn(\det_{\mathscr{C}}\left(u_1(z),\ldots, u_m(z)\right))}$$
where $(u_1(z),\ldots, u_m(z))\in U(\widetilde{E};\mathscr{O}_L[t_1,\ldots, t_s,z])$ is an $\widetilde{A_s}$-basis of the $z$-unit module. 
\end{theoreme}
\begin{proof}The proof of \cite[Corollary 7.5.6]{F} is still valid in the multi-variable setting at the integral level. We omit the details.
\end{proof}

The objective of this section is to prove that the $P$-adic $L$-series $L_P(\widetilde{E}/\widetilde{\mathscr{O}_{L,s}})$ converges in $\T_{s,z}(K_P)$.\\
Set $\Omega_{s,z}=\{x\in \T_{z,s}(L_P)^d \ens v_P(x)\geq 0\}$ and $\Omega_{s,z}^+=\{x\in \T_{z,s}(L_P)^d \ens v_P(x)> 0\}$.\\
Following the proof of Proposition \ref{T1} we have the two following convergences:
$$\log_{\widetilde{E},P}:\Omega_{s,z}^+\rightarrow \T_{s,z}(L_P)^d$$ and
$$\log_{\widetilde{F},P}:\Omega_{s,z}\rightarrow \T_{s,z}(L_P)^d.$$
We deduce that the $P$-adic $L$-series $L_P(\widetilde{E} /\widetilde{\mathscr{O}_{L,s}})$ is written in the form $\dfrac{w}{f}$ with $w\in \T_{z,s}(K_P)$ and $f\in \F_q[t_1,\ldots,t_s,z]$. We then consider $\zeta=(\zeta_1,\ldots,\zeta_s)\in \overline{\F_q}^s$ and we want to prove that we can evaluate the $P$-adic $L$-series at $t_i=\zeta_i$ for all $i=1,\ldots,s$ and at $z=\zeta\in \overline{\F_q}$ (simultaneously).

We use arguments very similar to those used for the convergence of the $P$-adic $L$-series, so we omit some of the details.

We set $\mathscr{K}(s)=\F_q(\zeta_1)\otimes_{\F_q}\ldots \otimes_{\F_q} \F_q(\zeta_s)$.
We then consider the following notation for $j=0,\ldots,s$:\\
$\bullet$ $k_j=\F_q(t_{j+1},\ldots,t_s)$, e.g., $k_0=k=\F_q(t_1,\ldots,t_s)$ and $k_s=\F_q$, \\
$\bullet$ $k_jA=k_j\otimes_{\F_q} A\simeq \F_q(t_{j+1},\ldots, t_s)[\theta]$,\\
$\bullet$ $k_jK=k_j\otimes_{\F_q}K\simeq\F_q(t_{j+1},\ldots, t_s,\theta)$ and $\widetilde{k_jK}=\F_q(z)\otimes_{\F_q}k_jK\simeq \F_q(z,t_{j+1},\ldots, t_s,\theta)$,\\
$\bullet$ $k_j\oo=k_j\otimes_{\F_q} \oo$,\\ 
$\bullet$ $A_{s,j}=\mathscr{K}(s)\otimes_{\F_q} k_jA,$\\
$\bullet$ $\widetilde{A_{s,j}}=\mathscr{K}(s)\otimes_{\F_q} k_j\widetilde{A}$,\\
$\bullet $ $\mathscr{O}_{L,s,j}= \mathscr{K}(s)\otimes_{\F_q} k_j\mathscr{O}_{L}$,\\ 
$\bullet$ $\widetilde{\mathscr{O}_{L,s,j}}=\mathscr{K}(s)\otimes_{\F_q} k_j\widetilde{\mathscr{O}_{L}},$\\
$\bullet$ For a place $w$ of $K$ extended to $k_jK$, $\widetilde{K(j)_{w}}$ is the completion of $\widetilde{k_jK}$ with respect to $w$.\\
$\bullet$ $\widetilde{L(j)_{w}}=L\otimes_K K(j)_w$,\\
$\bullet$ $\widetilde{M_{s,j,w}}=\mathscr{K}(s)\otimes_{\F_q}\T_{t_j}\left(L(j)_w\right)$,\\
$\bullet$ $\widetilde{L_{j,w}}=L\otimes_K \widetilde{K_{j,w}}$. \\ 
$\bullet$ For a place $w$ of $K$, $\T_{z,j}(L_w)=\T_{z,t_{j+1},\ldots,t_s}(L_w),$ e.g;, $\T_{z,0}(L_w)=\T_{z,t_1,\ldots,t_s}(L_w)$ and $\T_{z,s}(L_w)=\T_z(L_w)$.

For all $j=0,\ldots,s$, we extend the Frobenius $\tau$ into $\tau_s$ on $A_{s,j}$ by $\tau_s=\operatorname{id}\otimes \tau$ where $\operatorname{id}$ is the identity on $\mathscr{K}(s)$. We do the same for $\widetilde{A_{s,j}}$, for $\mathscr{O}_{L,s,j}$ and for $\widetilde{\mathscr{O}_{L,s,j}}$.

For $j=1,\ldots,s$ we define $E^{(j)}$ the homomorphism of $\mathscr{K}(s)\otimes_{\F_q}k_j$-algebras $E^{(j)}:A_{s,j}\rightarrow M_d(\mathscr{O}_{L,s,j})\{\tau_s\}$, that we call Anderson $A_{s,j}$-module defined over $\mathscr{O}_{L,s,j}$, by:
$$E^{(j)}_\theta=\Sum\limits_{i=0}^{r}\operatorname{ev}_{t_1=\zeta_1,\ldots, t_j=\zeta_j}(a_i) \tau_s^i$$
if $E_\theta=\Sum\limits_{i=0}^{r}a_i \tau_s^i\in M_d(\mathscr{O}_{L,s})\{\tau\}$.
We also set $E^{(0)}=E$ where we identify $a_i$ with $1\otimes a_i\in \mathscr{K}(s)\otimes_{\F_q} \mathscr{O}_{L}[t_1,\ldots, t_s] $ and replace $\tau$ with $\tau_s$.\\
Similarly, we define $\widetilde{E^{(j)}}$ the $z$-twist of $E^{(j)}$, that is the homomohprism of $\mathscr{K}(z)\otimes_{\F_q}k_j(z)$-algebras $\widetilde{E^{(j)}}:\widetilde{A_{s,j}}\rightarrow M_d(\widetilde{\mathscr{O}_{L,s,j}})\{\tau_s\}$, that we call Anderson $\widetilde{A_{s,j}}$-module defined over $\widetilde{\mathscr{O}_{L,s,j}}$, by:
$$\widetilde{E^{(j)}}_\theta=\Sum\limits_{i=0}^{r}\operatorname{ev}_{t_1=\zeta_1,.., t_j=\zeta_j}(a_i)z^i \tau_s^i$$

Finally, we consider $F=P^{-1}EP$ and construct $\widetilde{F^{(j)}}$ and ${F}^{(j)}$ in the same way.

\begin{lem}
    Consider $j\in \{1,\ldots,s\}$.
    \begin{enumerate}
 
      \item For all $a\in A_{s,j}$ we have the following equalities in $M_d(\mathscr{O}_{L,s,j})\{\tau_s\}$:
        $$E^{(j)}_a=\operatorname{ev}_{t_1=\zeta_1,\ldots, t_j=\zeta_j}E_a=\operatorname{ev}_{t_j=\zeta_j}E_a^{(j-1)}.$$
        \item We have the following equalities in $M_d\left(\mathscr{K}(s)\otimes_{\F_q}\F_q[t_{j+1},\ldots,t_s]L\right)\{\{\tau_s\}\}$:
        $$\exp_{E^{(j)}}=\operatorname{ev}_{t_1=\zeta_1,\ldots,t_j=\zeta_j}\exp_{E}=\operatorname{ev}_{t_j=\zeta_j}\exp_{E^{(j-1)}}.$$
        \item We have the following equalities for all $x$ in $\left(\mathscr{K}(s)\otimes_{\F_q}\T_{t_j}\left((\F_q(t_{j+1},\ldots,t_s)L)_{\infty}\right)\right)^d$:
        $$\exp_{E^{(j)}} (x)= \operatorname{ev}_{t_1=\zeta_1,\ldots, t_j=\zeta_j} \exp_{E} (x)=\operatorname{ev}_{t_j=\zeta_j} \exp_{E^{(j-1)}} (x).$$
    \end{enumerate}
\end{lem}
\begin{proof} It follows from definitions of the onjects, we omit the proof.
\end{proof}

We then define for all $j=1,\ldots,s$:
$$U(j)=\operatorname{ev}_{t_j=\zeta_j} U\left(\widetilde{E}^{(j-1)};\widetilde{\mathscr{O}_{L,s,j}}[t_j]\right)
\subseteq U\left(\widetilde{E}^{(j)};\widetilde{\mathscr{O}_{L,s,j}}\right). $$
Following the same arguments we used to prove Theorem \ref{Iso}, we have the following result.
  
\begin{theoreme}

\begin{enumerate}
    \item For all $j=1,\ldots, s$, we have an $\widetilde{A_{s,j}}$-module isomorphism:
    $$\dfrac{U(j)}{ U\left(\widetilde{E}^{(j)};\widetilde{\mathscr{O}_{L,s,j}}\right)}\simeq H\left(\widetilde{E}^{(j-1)};\widetilde{\mathscr{O}_{L,s,j}}[t_j]\right)[t_j-\zeta_j] $$ given by
  $$f_j(x)=\dfrac{\exp_{\widetilde{E}^{(j-1)}} x-\exp_{\widetilde{E}^{(j)}} x}{t_j-\zeta_j}$$
where $H\left(\widetilde{E}^{(j-1)};\widetilde{\mathscr{O}_{L,s,j}}[t_j]\right)[t_j-\zeta_j]$ is the $(t_j-\zeta_j)$-torsion of the class module  
$$ H\left(\widetilde{E}^{(j-1)};\widetilde{\mathscr{O}_{L,s,j}}[t_j]\right)=\dfrac{ \widetilde{E^{(j-1)}}\left(\widetilde{M_{s,j,\infty}}\right)}{\widetilde{E^{(j-1)}}\left(\widetilde{\mathscr{O}_{L,s,j}}[t_j]\right)+\exp_{\widetilde{E^{(j-1)}}}\left(\widetilde{E^{(j-1)}}(\widetilde{M_{s,j,\infty}})\right)}.$$
    \item The module $U(j)$ is a sub-$\widetilde{A_{s,j}}$ lattice of $U\left(\widetilde{E}^{(j)};\widetilde{\mathscr{O}_{L,s,j}}\right)$.
    \end{enumerate}
\end{theoreme}

We are now able to prove the main theorem of this section.

\begin{theoreme}\label{main2} The $P$-adic $L$-series does not have a pole in $\overline{\F}_q^s$. In other words we have: 
$$L_P({\widetilde{E}}/{{\mathscr{O}}}_{L,s})\in \T_{z,s}(K_P).$$
\end{theoreme}

\begin{proof}
We closely follow the proof of Theorem \ref{th:paspole}.
We identify $\mathscr{C}=(g_1,\ldots,g_m)$ with $(1\otimes g_1,\ldots,1\otimes g_m)\subseteq \mathscr{K}(s)\otimes \operatorname{Lie}_E(\oo)$.
Consider $(u_{i,1})_{i=1,\ldots,m}\subseteq U(\widetilde{F};\mathscr{O}_L[t_1,\ldots,t_s,z])$ an $\widetilde{A}_s$-basis of $U(\widetilde{F};\widetilde{\mathscr{O}_{L,s}})$. Set $$w_1=\det_{\mathscr{C}}(u_{1,1},\ldots, u_{m,1})\in \T_{z,s}(K_\infty)$$ with sign $$f_1\in \F_q[z,t_1,\ldots,t_s]$$ and set $$w_{1,P}=\det_{\mathscr{C}}(\log_{\widetilde{F},P} (\exp_{\widetilde{F}} (u_{1,1})),\ldots,\log_{\widetilde{F},P} (\exp_{\widetilde{F}} (u_{m,1})))\in \T_{z,s}(K_P).$$
We want to prove that the quotient $\dfrac{w_{1,P}}{f_1}$ is an element of $\T_{s,z}(K_P)$. We will prove that we can evaluate the last quantity at every $\zeta=(\zeta_1,\ldots,\zeta_s)\in \overline{\F}_q^s$ and at $z=\zeta\in \overline{\F_q}$.\\

We will prove by induction that for all $k=1,\ldots, s$,
there exists $(v_{i,k+1}, i=1,\ldots, m)$ an $\widetilde{A_{s,k}}$-basis of $U(k)$ and $x_{k+1}\in \mathscr{K}(s)\otimes_{\F_q}\widetilde{A_{s,k}}$ such that
$$\operatorname{ev}_{t_1=\zeta_1,\ldots, t_k=\zeta_k}(1\otimes L(\widetilde{F}/\widetilde{\mathscr{O}_{L,s}})=\dfrac{\det_{\mathscr{C}}(v_{i,{k+1}}, i=1,\ldots, m)}{x_{k+1}}$$ and
$$\begin{aligned}\operatorname{ev}_{t_1=\zeta_1,\ldots, t_k=\zeta_k}(1\otimes \dfrac{w_{1,P}}{f_1})&=\dfrac{\det_{\mathscr{C}}(\log_{\widetilde{F}^{(k)},P} \exp_{\widetilde{F}^{(k)}}v_{i,{k+1}}, i=1,\ldots, m)}{x_{k+1}}\\&\in \dfrac{1}{x_{k+1}}\mathscr{K}(s)\otimes_{\F_q}\T_z(t_{k+1},\ldots, t_s)(L_P).\end{aligned}$$ 
 \underline{Step 1}: evaluation at $t_1=\zeta_1$.\\ 
By Theorem \ref{thm:classeS} we have
$$L(\widetilde{F}/\widetilde{\mathscr{O}_{L,s}})=\dfrac{w_1}{f_1}\in \T_{z,s}(K_\infty).$$
Consider$(v_{i,2})_{i=1,\ldots,m}$ an $\widetilde{A_{s,1}}$ basis of $U(1)$ that can be assumed to be at the entire level, i.e., $(v_{i,2})\subseteq \left(\mathscr{K}(s)\otimes_{\F_q}\T_{z,t_2,\ldots,t_s}(L_\infty)\right)^d$ and let $(u_{i,2})\subseteq \left(\mathscr{K}(s)\otimes \T_{z,s}(L_\infty)\right)^d$ be above (i.e., $\operatorname{ev}_{t_1=\zeta_1} u_{i,2}=v_{i,2})$ for all $i$). Set $$w_2=\det_{\mathscr{C}} (u_{1,2},\ldots, u_{m,2})\in \mathscr{K}(s)\otimes_{\F_q}\T_{z,s}(K_\infty)$$ that is not divisible by $t_1-\zeta_1$ and set
$$w_{2,P}=\det_{\mathscr{C}} (\log_{\widetilde{F},P}(\exp_{\widetilde{F}} (u_{1,2})),\ldots, \log_{\widetilde{F},P}(\exp_{\widetilde{F}} (u_{m,2}))) \in \mathscr{K}(s)\otimes_{\F_q}\T_{z,s}(K_P)$$ the $P$-adic analog of $w_2$. Set $\delta_2=\det_{( 1 \otimes u_{1,1},\ldots, 1\otimes_{u_{m,1}})}(u_{1,2},\ldots, u_{m,2})\in \mathscr{K}(s)\otimes_{\F_q}\widetilde{A}_s$. We have:$$1\otimes_{\F_q} w_1=\dfrac{w_2}{\delta_2}$$ and
$$ 1\otimes w_{1,P}=\dfrac{w_{2,P}}{\delta_2}.$$
We deduce the following equality from the class formula:
$$ 1\otimes L(\widetilde{F}/\widetilde{\mathscr{O}_{L,s}})=\dfrac{w_2}{\delta_2( 1\otimes f_1)}.$$
Since $\zeta_1$ is not a pole of the $L$-series and $t_1- \zeta_1$ does not divide $w_2$, we deduce that we can evaluate $\delta_2(1\otimes f_1)$ at $t_1=\zeta_1$ in a non-zero element $x_2$ of $\mathscr{K}(s)\otimes_{\F_q}\widetilde{A}_{s-1}$, in other words
$$\operatorname{ev}_{t_1=\zeta_1}1\otimes L(\widetilde{F}/\widetilde{\mathscr{O}_{L,s}})=\dfrac{\det_{\mathscr{C}} (v_{i,2})}{x_2}\in \dfrac{1}{x_2}\mathscr{K}(s)\otimes_{\F_q}\T_{z,t_2,\ldots,t_s}(L_\infty).$$
Next, from the equality
$$1\otimes \dfrac{w_{1,P}}{f_1}=\dfrac{w_{2,P}}{\delta_2(1\otimes f_1)}$$ we deduce that we can evaluate the $P$-adic $L$-series at $t_1=\zeta_1$:
$$\begin{aligned}\operatorname{ev}_{t_1=\zeta_1}  1\otimes \dfrac{w_{1,P}}{f_1}&=\dfrac{\det_{\mathscr{C}}\left(\log_{\widetilde{F}^{(1)},P}(\exp_{\widetilde{F}^{(1)}} (v_{i,2})),i=1,\ldots,m\right)}{x_2}\\
&\in \dfrac{1}{x_2} \mathscr{K}(s)\otimes_{\F_q} \T_{z,t_2,\ldots,t_s}(L_P).\end{aligned}$$
\underline{Step $k$}: Assume the result to be true up to rank $s-1\geq k-1\geq 1$ and we prove it at rank $k$. 
Consider $(v_{i,{k+1}})_{i=1,\ldots,m}$ a $\widetilde{A_{s,k}}$ basis of $U(k)$ that can be assumed to be to the entire level, i.e., $v_{i,{k+1}}\subseteq \left(\mathscr{K}(s)\otimes_{\F_q}\T_{z,t_{k+1},\ldots, t_s}(L_\infty)\right)^d$ and let$(u_{i,k+1})\subseteq \left(\mathscr{K}(s)\otimes_{\F_q} \T_{z,t_k,\ldots, t_s}(L_\infty)\right)^d$ be above (i.e., $\operatorname{ev}_{t_k=\zeta_k} u_{i,k+1}=v_{i,k+1}$).\\
Set
$$w_{k+1}=\det_{\mathscr{C}}(u_{1,{k+1}},\ldots, u_{m,k+1})\in \mathscr{K}(s)\otimes_{\F_q}\T_{z,t_k,\ldots, t_s}(K_\infty)$$
that is not divisible by $t_k-\zeta_k$ and set
$$w_{k+1,P}=\det_{\mathscr{C}}( \log_{\widetilde{F}^{(k-1)}}(\exp_{\widetilde{F}^{(k-1)}}(u_{i,k+1})),i=1,\ldots,m)\in \mathscr{K}_s\otimes_{\F_q} \T_{z,t_k,\ldots, t_s}(K_P)$$ the $P$-adic analog of $w_{k+1}$. Set $\delta_{k+1}=\det_{(v_{1,k}),\ldots, v_{m,k}}(u_{1,k+1},\ldots, u_{m,k+1})\in \mathscr{K}(s)\otimes \widetilde{A_{s,k-1}}.$ We have:
$$\det_{\mathscr{C}}(v_{1,k},\ldots, v_{m,k})=\dfrac{\det_{\mathscr{C}}(u_{1,k+1},\ldots, u_{m,k+1})}{\delta_{k+1}}$$ and
$$\det_{\mathscr{C}}\left(\log_{\widetilde{F}^{(k-1),P}}(\exp_{\widetilde{F}^{(k-1)}} (v_{1,k})),\ldots, \log_{\widetilde{F}^{(k-1),P}}(\exp_{\widetilde{F}^{(k-1)}} (v_{m,k}))\right)=\frac{w_{k+1,P}}{\delta_{k+1}}.$$
Then we have the following equalities:

$$\operatorname{ev}_{t_1=\zeta_1,\ldots, t_k=\zeta_k}\left( 1\otimes L(\widetilde{F}/\widetilde{\mathscr{O}_{L,s}})\right)=\dfrac{w_{k+1}}{x_k\delta_{k+1}}$$ and
$$\operatorname{ev}_{t_1=\zeta_1,\ldots, t_k=\zeta_k}\left( 1\otimes \dfrac{w_{1,P}}{f_1}\right)=\dfrac{w_{k+1,P}}{x_k\delta_{k+1}}.$$
Since we can evaluate at $t_k=\zeta_k$ the $L$-series and $t_k-\zeta_k$ does not divide $w_{k+1}$, we can evaluate $x_k\delta_{k+1}$ at $t_k=\zeta_k$ into a non-zero element $x_{k+1}\in \mathscr{K}(s)\otimes \widetilde{A_{s,k}}$. We have:
$$\operatorname{ev}_{t_k=\zeta_k,\ldots, t_z=\zeta_1}( 1\otimes L(\widetilde{F}/\widetilde{\mathscr{O}_{L,s}}))=\operatorname{ev}_{t_k=\zeta_k}\dfrac{w_{k+1}}{x_k\delta_{k+1}}=\dfrac{\det_{\mathscr{C}}(v_{i,{k+1}},i=1,\ldots, m)}{x_{k+1}} $$ and
$$\begin{aligned}\operatorname{ev}_{t_k=\zeta_k,\ldots,  t_1=\zeta_1}1\otimes \dfrac{w_{1,P}}{f_1}&=\dfrac{\det_{\mathscr{C}}\left(\log_{\widetilde{F}^{(k)},P}(\exp_{\widetilde{F}^{(k)}}(v_{i,k+1}))\right)}{x_{k+1}}\\ &\in \dfrac{1}{x_{k+1}}\mathscr{K}(s)\otimes_{\F_q} \T_{z,t_{k+1},\ldots,t_s}(K_P).\end{aligned}$$

\underline{Last step}: evaluation at $z$. We write $z=t_{s+1}$ and use similar arguments to conclude.

\end{proof}

\begin{rem}
   If at some step $j\in \{1,\ldots, s\}$ we have for all $i=1,\ldots, r$: 
$$\operatorname{ev}_{t_1=\zeta_1, \ldots, t_j=\zeta_j} A_i=0,$$ then we have for all $k\geq j$:
$$E_{\theta}^{(k)}=\theta I_d+N_k$$ with $N_k\in M_d(\mathscr{K}(s)\otimes_{\F_q} \F_q[t_{k+1},\ldots,t_s]\oo)$ a nilpotent matrix.
Then we have:
$$\partial_{E^{(k)}}(\theta^{q^d})=E^{(k)}_{\theta^{q^d}}=\theta^{q^d}I_d.$$
Hence, if $\exp_{E^{(k)}}=\Sum\limits_{n\geq 0}d_{n,k}\tau_s^k$, then from the functional equation of the exponential map we obtain for all $n\geq 0$:
$$(\theta^{q^d})^{q^n}d_{n,k}=\theta^{q^d}d_{n,k}.$$
Thus $d_{0,k}=I_d$ and $d_{n,k}=0$ for all $n\geq 1$ so finally $\exp_{E^{(k)}}=I_d\tau^0$ for all $k\geq j$. Then we have $U(\widetilde{E}^{(j)};\widetilde{\mathscr{O}_{L,s,j}})=\operatorname{Lie}_{\widetilde{E}^{(j)}}(\widetilde{\mathscr{O}_{L,s,j}})$ for all $j\geq k$, and the previous proof is still valid.
\end{rem}

\section{Applications}\label{section:facile}
In this section we investigate the case where $L=K$ and $d=1$, i.e., the case of Drinfeld modules defined over the ring $A$ itself.
\subsection{Preliminaries}

We consider a Drinfeld module $\phi:A\rightarrow A\{\tau\}$ of rank $r\geq 1$ and we denote by $\exp_\phi=\Sum\limits_{n\geq 0} d_n\tau^n\in K\{\{\tau\}\}$ its associated exponential map. We call a period of $\phi$ any element $\lambda\in \C_\infty$ such that $\exp_\phi(\lambda)=0$. We denote by $\Lambda_\phi$ its period lattice which is a free $A$-module of rank $r$. Let $NP(\phi)$ be the Newton polygon associated with $\exp_\phi$ which is defined as the lower convex hull of the points $P_n=(q^n-1,v_\infty(d_n))$). Remark that the zeros of $\exp_\phi$ are all simple since $\dfrac{d}{dx} \exp_\phi x=1$.\\
We have the following property about the edges of $NP(\phi)$ that can be found in \cite[Theorem 2.5.2]{papik}.

\begin{prop}
Consider $\lambda$ a non-zero period of $\phi$ of valuation $x$, and let $N$ be the number of periods of valuation equal to $x$. Then $NP(\phi)$ has a single edge with slope $\lambda$ and length $N$.
\end{prop}

By Lemma \ref{lem:expinfini}, we know that $\lim\limits_{n\rightarrow +\infty} v_\infty(d_n)=+\infty$, then we define $N_0$ as the smallest $n$ such that $v_\infty(d_n)$ is minimal, and $N_1$ as the largest $n$ such that $v_\infty(d_n)$ is minimal. Another way of looking at $N_0$ and $N_1$ is that the edge of slope equals to $0$ of $NP(\phi)$ has endpoints $P_{N_0}$ and $P_{N_1}$.

We will use the concept of successive minimum basis from \cite[section 3]{tower}.

\begin{defi}
    An ordered $A$-basis $(\lambda_1,\ldots,\lambda_r$) of the $A$-lattice $\Lambda_\phi$ in $\C_\infty$ is a successive minimum basis ( shortly an SMB) if for each $1\leq i\leq r$, the vector $\lambda_i$ has minimal valuation $v_\infty(\lambda_i)$ among all $w\in \Lambda_\phi$ not in the
span $\Sum\limits_{1\leq j<i} \lambda_i A$ of $\{\lambda_1,\ldots, \lambda_{i-1}\}$.
\end{defi}

Gekeler proved the following result, see \cite[Proposition 3.1]{tower}
\begin{prop}\label{gek}
\begin{enumerate}
    \item The period lattice $\Lambda_\phi$ admits an SMB. 
    \item The sequence $(v_\infty(\lambda_i))_{i=1,\ldots,r}$ is independent of the choice of the SMB.
    \item Consider $\{\lambda_1,\ldots,\lambda_r\}$ an SMB for $\Lambda_\phi$. Then for all $\lambda=\Sum\limits_{i=1}^{r}a_i\lambda_i\in \Lambda_\phi$ we have:

$$v_\infty(\lambda)=\min \{ v_\infty(a_i\lambda_i) \ens i=1,\ldots,r\}.$$
\end{enumerate}
\end{prop}

We consider $s\in \{1,\ldots, r\}\cup \emptyset$ be such that $v_\infty(\lambda_i)\geq 0$ for $i=1,\ldots, s$ and $v_\infty(\lambda_i)<0$ for $i>s$. Denote also by $x_i=v_\infty(\lambda_i)$ for all $i=1,\ldots, r$.
 Then we consider $n_1,\ldots,n_t\in \{1,\ldots,s\}$ be such that $x_{n_i}\in \N$ for $i=1,\ldots,t$, and we denote by $S_\phi=\{n_1,\ldots,n_t\}$.

\begin{prop}\label{0}
    We have the following equality:
    $$N_0-N_1=t.$$
\end{prop}
\begin{proof}
First we calculate $N_0.$ To do this, we need to count the total length of the strictly negative slopes (which is also equal to $q^{N_0}-1$ by definition), in other words the number of non-zero periods of strictly positive valuation. Set $\lambda=\Sum\limits_{i=1}^{r}a_i\lambda_i\in \Lambda_\phi$. We have the equivalence by Proposition \ref{gek}:
$$v_\infty(\lambda)>0\Leftrightarrow \left\{\begin{aligned}&a_i=0 \text{ if } i>s, \\&\deg (a_i) \leq \floor{x_i} \text{ if } i\leq s \text{ and } i\notin S_\phi, \\&\deg (a_i)<x_i \text{ if } i \in S_\phi.\end{aligned}\right.$$
We finally exclude the case $\lambda=0$. We obtain that the total number of non-zero elements of $\Lambda_\phi$ with strictly positive valuation is equal to
$$q^{N_0}-1=\Prod\limits_{j=1}^{t}q^{x_{n_j}}\Prod\limits_{i\leq s, i\notin S_\phi} q^{\floor{x_i}+1}-1.$$ By applying the logarithm we obtain:
$$N_0=\Sum\limits_{j=1}^{t}x_{n_j}+\Sum\limits_{i\leq s, i\notin S_\phi} (\floor{x_j}+1).$$ 
We then calculate $q^{N_1}-1$ that is equivalent to counting the number of periods with positive valuation in a similar way, and we obtain 
$$N_1=\Sum\limits_{j=1}^{t}(x_{n_j}+1)+\Sum\limits_{i\leq s, i\notin S_\phi} (\floor{x_j}+1)=N_0+t.$$
\end{proof}

Note that we only worked with periods of $\phi$ and never used the fact that $\phi$ is defined over $A$. We can therefore generalize the previous result to any $\phi:A\rightarrow \C_\infty\{\tau\}$ of rank $r$ since the concept of SMB is defined in full generality.

We want to apply this result to the study of the vanishing order of $L_P(\widetilde{\phi}/\widetilde{A})$ at $z=1$.
\subsection{An application to the vanishing of the \texorpdfstring{$P$}{}-adic \texorpdfstring{$L$}{}-series}
 Let $P\in A$ be irreducible monic and set $u_\phi(z)=\exp_{\widetilde{\phi}} L(\widetilde{\phi}/\widetilde{A})\in A[z]$.
We set:
$$g_{P,\phi}(z)=\left[\widetilde{\phi}(\widetilde{A}/P\widetilde{A})\right]_{\widetilde{A}}\in A[z]$$ and
$$g_{P,\phi}(1)=\left[{\phi}({A}/P{A})\right]_{{A}}\in A\backslash\{0\}.$$
We recall the definition of local factor associated with $\widetilde{\phi}$ and $P$:
$$z_P(\widetilde{\phi}/\widetilde{A})=\dfrac{P}{g_{P,\phi}(z)}.$$
Let us recall that the $P$-adic $L$-series associated with $\phi$ is defined as follows:

$$L_P(\widetilde{\phi}/\widetilde{A})=\dfrac{1}{P}\log_{\widetilde{\phi},P} \widetilde{\phi}_{g_{P,\phi}(z)}( u_\phi(z))\in \T_z(K_P).$$ 
By the proof of \cite[Corollary 7.5.6]{F}, we deduce that $L(\widetilde{\phi}/\widetilde{A})\in \T_z(K_\infty)$ is a unit in $\T_z(K_\infty)$ whose valuation is equal to $0$ and whose constant coefficient is equal to $1$.

From now on, we say that $\phi$ does not have $A$-torsion if the $A$-module $\phi(A)$ is torsion-free.
 \begin{prop}\label{2}
     Let $\phi$ be an $A$-Drinfeld module defined over $A$ of rank $r\geq 1$ without $A$-torsion. Then for all $k\geq 0$ the following assertions are equivalent:
    \begin{enumerate}
        \item $(z-1)^k \vert_{\T_z(K_P)} L_P(\widetilde{\phi}/\widetilde{A})$,
        \item $(z-1)^k \vert_{A[z]} u_\phi(z)$.
    \end{enumerate}
 \end{prop}

\begin{proof}
By the definition of the $P$-adic $L$-series:
$$L_P(\widetilde{\phi}/\widetilde{A})=\dfrac{1}{P}\log_{\widetilde{\phi},P} (\widetilde{\phi}_{g_{P,\phi}(z)}( u_\phi(z))),$$
    we see that $2\Rightarrow 1$ is clear.
    
    Let us prove $1\Rightarrow  2$. We have:
    $$P^2L_P(\widetilde{\phi}/\widetilde{A})=\log_{\widetilde{\phi},P} (\widetilde{\phi}_{Pg_{P,\phi}(z)}( u_\phi(z))).$$
    Since $v_P(\widetilde{\phi}_{Pg_{P,\phi}(z)}( u_\phi(z)))\geq 2$, we have $\widetilde{\phi}_{Pg_{P,\phi}(z)}( u_\phi(z))\in \mathcal{D}_z^+$. By applying the $P$-adic exponential map we obtain:
    $$\exp_{\widetilde{\phi},P} (P^2L_P(\widetilde{\phi}/\widetilde{A}))=\widetilde{\phi}_{Pg_{P,\phi}(z)}( u_\phi(z))\in A[z].$$
If $(z-1)^k \vert_{\T_z(K_P)} L_P(\widetilde{\phi}/\widetilde{A})$, then we have 
$(z-1)^k\vert_{A[z]} \widetilde{\phi}_{Pg_{P,\phi}(z)}( u_\phi(z))$. Since $\phi$ does not have $A$-torsion, we deduce that $$(z-1)^k\vert_{A[z]}  u_\phi(z).$$

\end{proof}

\begin{prop}\label{3} 
  Let $m\in A\backslash\{0\}$ be a non zero polynomial and consider the Drinfeld module $\psi=m^{-1}\phi m$.
Then the vanishing order at $z=1$ of $L_P(\widetilde{\phi}/\widetilde{A})$ and $L_P(\widetilde{\psi}/\widetilde{A})$ are equal.
\end{prop}

\begin{proof}
    By Lemma \ref{tordulocal}, we have the following equality in $\T_z(K_P)$:
    $$g_{P,\phi}(z)L_P(\widetilde{\psi}/\widetilde{A})=g_{P,\psi}(z)L_P(\widetilde{\phi}/\widetilde{A})\Prod\limits_{Q\vert m }\dfrac{g_{Q,\phi}(z)}{Q}.$$
    Since $g_{Q,\phi}(1)\neq 0$ for all $Q$, we obtain the result.
\end{proof}

\begin{prop}\label{5}
    If $\phi:A\rightarrow A\{\tau\}$ has rank $r$, then the vanishing order at $z=1$ of the $P$-adic $L$-series $L_P(\widetilde{\phi}/\widetilde{A})$ is lower than or equal to $r$.
\end{prop}

\begin{proof}
Let us first twist $\phi$ into $\psi=m^{-1}\phi m$ without $A$-torsion. By Proposition \ref{2} we consider the vanishing order at $z=1$ of $u_\psi(z)\in A[z]$. We can compute its leading coefficient seen as a polynomial in the variable $\theta$. We have $u_\psi(z)=\exp_{\widetilde{\psi}} L(\widetilde{\psi}/\widetilde{A})=\Sum\limits_{n\geq 0} d_n z^n \tau^n(L(\widetilde{\psi}/\widetilde{A}))$. We know that $L(\widetilde{\psi}/\widetilde{A})$ has the form $1+\Sum\limits_{n\geq 1} a_nz^n \in \T_z(K_\infty)$ with $v_\infty(a_n)>0$. Let $N_0,m_1,\ldots,m_l,N_1$ be the integers $n$ such that $v_\infty(d_n)$ is minimal.
 Let $\beta_{n}\in \F_q^\ast$ be the sign of $d_{n}$, we obtain:
 $$\operatorname{sgn} (u_\psi(z))=z^{N_0}(\beta_{N_0}+\ldots +\beta_{N_1}z^{N_1-N_0})\in \F_q[z]$$ that has at most $r$ non-zero roots by Proposition \ref{0}. Thus, $u_\psi(z)$ is divisible at most by $(z-1)^r$. 
\end{proof}

Note that we have proved more precisely: 
$$\ord_{z=1}L_P(\widetilde{\phi}/\widetilde{A})\leq\#\{i=1\ldots,r \ens v_\infty(\lambda_i)\in \Z\}.$$

\begin{prop} The previous inequality is not an equality in general.\end{prop}

\begin{proof}
Consider the Drinfeld module given by $\phi_\theta=\theta+\theta\tau^2$ with $q=3$. One can prove that the Newton polygon of the associated exponential map is the polygon beginning at the point $(0,0)$ and has successive slopes of length $(q^{2k+2}-q^{2k})$ and equal to $k+1$. Thus, the number of periods of an SMB having valuation $\in\Z$ is equal to $2$. By \cite[Proposition 2.21]{cgl} we have $u_\phi(1)=1$. One can prove that $\phi$ does not have $A$-torsion. Then by Proposition \ref{2} we obtain $\operatorname{ord}_{z=1} L_P(\widetilde{\phi}/\widetilde{A})=0$.
\end{proof}
\begin{rem} For any $r\geq 1$, we can construct explicit Drinfeld modules of rank $r$ whose vanishing order of the associated $P$-adic $L$-series equals $r$. In fact, denote by $(-1)^r (z-1)^r= 1 +\Sum\limits_{i=1}^r \alpha_i z^i$, with $\alpha_i\in \F_q$ and consider the Drinfeld module given by $\phi_\theta=\theta+\Sum\limits_{i=1}^r \alpha_i \theta^{q^i}\tau^i$. By \cite[Proposition 2.21]{cgl}, we have: $u_\phi(z)=1 +\Sum\limits_{i=1}^r \alpha_i z^i=(-1)^r (z-1)^r$. Then the vanishing order at $z=1$ is greater than or equal to $r$, so equals $r$.
\end{rem}
To conclude the text, we would like to ask the following question from a personal communication with Xavier Caruso and Quentin Gazda.
\begin{ques} Do we have the following equality
$$\operatorname{ord}_{z=1}L_P(\widetilde{\phi}/\widetilde{A})=\# \left\{ i\in\{1,\ldots, r\}\ens \lambda_i \in \bigcup\limits_{n\geq 0} \F_{q^{p^n}}((\frac{1}{\theta}))\right\} ?$$
\end{ques}

\bibliographystyle{abbrv}
\bibliography{ref}

\end{document}